\documentclass[a4paper, 10pt]{article}\date{}\textheight=27.2cm \voffset=-3.6cm \textwidth=18cm \hoffset=-3cm
\usepackage{array,tabularx,amsthm,amsmath,amssymb,mathrsfs,mathtools,picture,graphicx,color,hyperref}
\hypersetup{colorlinks=true,linkcolor=red,citecolor=blue,filecolor=magenta,urlcolor=cyan}

\allowdisplaybreaks

\def\_email#1@#2\q_nil{\href{mailto:#1@#2}{{\emailfont #1\emailampersat #2}}}
\newcommand\emailampersat{{\color{cyan}\small@}}

\newtheorem{thm}{Theorem}[section]
\newtheorem{lma}[thm]{Lemma}

\newtheorem{rmk}{Remark}
\newtheorem{prop}[thm]{Proposition}
\newtheorem{claim}{Claim}
\numberwithin{equation}{section}

\title{On the asymptotic limit of steady state Poisson--Nernst--Planck equations with steric effects}

\author{Jhih-Hong Lyu \thanks{Department of Mathematics, National Taiwan University, Taipei 10617, Taiwan (\tt d06221001@ntu.edu.tw).}, Tai-Chia Lin \thanks{Department of Mathematics, National Taiwan University, Taipei 10617, Taiwan; National Center for Theoretical Sciences, Mathematics Division, Taipei 10617, Taiwan ({\tt tclin@math.ntu.edu.tw}).}}

\begin{document}
\maketitle
\begin{abstract}
When ions are crowded, the effect of steric repulsion between ions becomes significant and the conventional Poisson--Boltzmann (PB) equation (without steric effect) should be modified.
For this purpose, we study the asymptotic limit of steady state Poisson--Nernst--Planck equations with steric effects (PNP-steric equations).
By the assumptions of steric effects, we transform steady state PNP-steric equations into a PB equation with steric effects (PB-steric equation) which has a parameter $\Lambda$ and positive constants $\lambda_i$'s (depend on the radii of ions and solvent molecules).
The nonlinear term of PB-steric equation is mainly determined by a Lambert type function which represents the concentration of solvent molecules. As $\Lambda=0$, the PB-steric equation becomes the conventional PB equation but as $\Lambda>0$, a large $\Lambda$ makes the steric repulsion (between ions and solvent molecules) stronger.
This motivates us to find the asymptotic limit of PB-steric equation as $\Lambda$ goes to infinity.
Under the Robin (or Neumann) boundary condition, we prove theoretically and numerically that the PB-steric equation has a unique solution $\phi_\Lambda$ which converges to solution $\phi^*$ of a modified PB (mPB) equation as $\Lambda$ tends to infinity.
Our results show that the limiting equation of PB-steric equation (as $\Lambda$ goes to infinity) is a mPB equation which has the same form (up to scalar multiples) as those mPB equations in \cite{1942bikerman,1997borukhov,2007kilic,2009li,2009li2,2013li,2011lu}.
Therefore, the PB-steric equation can be regarded as a generalized model of mPB equations.
\vspace{5mm}
\vspace{1mm}\end{abstract}

\section{Introduction}
\label{sec1}

As a well-known model of ion transport, the Poisson--Nernst--Planck (PNP) equations are important in the study of many physical and biological phenomena (cf. \cite{1995andelman,1997barcilon,2005coalson,1993eisenberg,2007eisenberg}).
The PNP equations represent ions as point particles without size.
However, in crowded ions, steric repulsions may appear due to ion sizes so the PNP equations should be modified (cf. \cite{2002abbas,2019aitbayev,2011bazant,2017bates,2021baude,2016chen,2019ding,2020ding,2011eisenberg,2018gavish,2008grochowski,2012horng,2017huang,2010kalcher,2007kilic,2007kornyshev,2020Park,2018siddiqua}).
To include the steric effect of ion sizes, we use the approximate Lennard--Jones potential to develop a new model called the Poisson--Nernst--Planck equations with steric effects (PNP-steric equations) which have parameters depending on ion radii (cf. \cite{2014lin}).
Recently, the PNP-steric equations have been used to study the selectivity and gating of ion channels \cite{2020ding2,2012horng,2019horng}, the anodic dissolution at high current densities \cite{2020kohn}, and the slit-shaped nanopore conductance  \cite{2019ding}.
Hence the PNP-steric equations become a useful model which describes ion transport with steric effects.

For the mixture of ions and solvent molecules, the PNP-steric equations can be denoted as
\begin{align}
\label{eq:1.01}
&\frac{\partial c_i}{\partial t}+\nabla\cdot\mathcal{J}_i=0,\quad\quad i=0,1,\dots,N,
\\\label{eq:1.02}
&\mathcal{J}_i=-D_i\nabla c_i-\frac{D_ic_i}{k_BT}z_i\text{e}\nabla\phi-\frac{D_ic_i}{k_BT}\sum_{j=0}^Ng_{ij}\nabla c_j,
\\\label{eq:1.03}&
-\nabla\cdot\left(\epsilon\nabla\phi\right)=\rho_0+\sum_{i=0}^Nz_i\text{e}c_i,
\end{align}
where $N$ is the number of ion species, $\mathcal{J}_i$ is the flux density, $D_i$ is the diffusion constant for $i=0,1,\dots,N$.
In addition, $c_0$ is the concentration of solvent molecules with the valence $z_0=0$, the radius $a_0$, and $c_i$ is the concentration of the $i$th ion species with the valence $z_i\neq0$ and the radius $a_i$ for $i=1,\dots,N$.
Moreover, $\phi$ is the electrostatic potential, $\epsilon$ is the dielectric function, $\rho_0$ is the permanent charge density function, $k_B$ is the Boltzmann constant, $T$ is the absolute temperature, $\text{e}$ is the elementary charge and \begin{equation}
\label{eq:1.04}
g_{ij}=S_{\sigma}\epsilon_{ij}(a_i+a_j)^{12}\geq0,
\end{equation}
where $g_{ij}$ represents the strength of steric repulsion, $\epsilon_{ij}$ is the energy coupling constant, and $S_{\sigma}\sim\sigma^{d-12}$ for the dimension $d\leq3$ (cf. \cite{2014lin}). 

To get the steady state of equations \eqref{eq:1.01}--\eqref{eq:1.03}, we set $\mathcal{J}_i=0$ for $i=0,1,\dots,N$.
Then equation \eqref{eq:1.02} implies
\begin{equation}
\label{eq:1.05}
D_i\ln c_i+\frac{D_i}{k_BT}z_i\text{e}\phi+\frac{D_i}{k_BT}\sum_{j=0}^Ng_{ij}c_j=\mu_i,\quad i=0,1,\dots,N,
\end{equation}
where $\mu_i$ is constant for $i=0,1,\dots,N$.
Without loss of generality, we set $D_i=k_BT=\text{e}=1$ and equations \eqref{eq:1.03} and \eqref{eq:1.05} become
\begin{align}
\label{eq:1.06}
&\ln c_i+z_i\phi+\sum_{j=0}^{N}g_{ij}c_j=\mu_i\quad\text{for}~i=0,1,\dots,N,\\
\label{eq:1.07}&-\nabla\cdot\left(\epsilon\nabla\phi\right)=\rho_0+\sum_{i=1}^Nz_ic_i\quad\text{in}~\Omega,\end{align}
where $\Omega$ is a bounded smooth domain in $\mathbb{R}^d$.
Under the assumptions of $g_{ij}$ and $\mu_i$ (see the assumptions of steric effects (A1) and (A2) below), equation \eqref{eq:1.06} has unique smooth solutions $c_i=c_i(\phi)$, for $\phi\in\mathbb R$ and $i=0,1,\dots,N$, so equation \eqref{eq:1.07} becomes a single nonlinear elliptic equation (see equation \eqref{eq:1.12}).
Notice that equation \eqref{eq:1.06} may have multiple solutions if $g_{ij}$'s satisfy the other (different from (A1)) conditions in \cite{2015lin}.
Hereafter, the boundary condition of equation \eqref{eq:1.07} is either the Robin boundary condition
\begin{equation}
\label{eq:1.08}
\phi+\eta\frac{\partial\phi}{\partial\nu}=\phi_{bd}\quad\text{on}~\partial\Omega,
\end{equation}or the Neumann boundary condition
\begin{equation}
\label{eq:1.09}
\frac{\partial\phi}{\partial\nu}=0\quad\text{on}~\partial\Omega,
\end{equation}where $\phi_{bd}\in\mathcal{C}^2(\overline{\Omega})$ is the extra electrostatic potential and $\eta\geq0$ is a constant related to the surface dielectric constant.

Volume exclusion is considered to develop the modified Poisson--Boltzmann (mPB) equation with steric effects (cf. \cite{1997borukhov}). When volume exclusion occurs, ions and solvent molecules are well separated.
This motivates us to assume that ions and solvent molecules have strong repulsion in order to separate ions and solvent molecules.
To describe strong repulsion of ions and solvent molecules, we set $g_{ij}=\Lambda\tilde{g}_{ij}$ for $i,j=0,1,\dots,N$, where $\Lambda$ is a large parameter tending to infinity and each  $\tilde{g}_{ij}$ is a positive constant independent of $\Lambda$.
The main goal of this paper is to study the asymptotic limit of equations \eqref{eq:1.06} and \eqref{eq:1.07} as $\Lambda$ goes to infinity.
The singularity of matrix $\tilde{G}=[\tilde{g}_{ij}]$ plays a crucial role on this problem.
Suppose that $\tilde{G}$ is nonsingular with inverse matrix $\tilde{G}^{-1}$ and ${\bf c}_{\Lambda}=[c_{0,\Lambda},\dots,c_{N,\Lambda}]^{\mathsf{T}}$ is a solution of equation \eqref{eq:1.06}, where $c_{i,\Lambda}$ is smooth function of $\phi$ for $i=0,1,\dots,N$.
Then we differentiate the equation \eqref{eq:1.06} with respect to $\phi$ and obtain $(D_{\Lambda}+\Lambda\tilde{G})\frac{\mathrm{d}{\bf c}_{\Lambda}}{\mathrm{d}\phi}=-{\bf z}$, where $D_{\Lambda}$ is a diagonal matrix with diagonal entries $(c_{i,\Lambda})^{-1}$ and ${\bf z}=[z_0,\cdots,z_N]^{\mathsf{T}}$.
Suppose that $D_{\Lambda}$ is uniformly bounded for $\Lambda\geq1$, and $c_{i,\Lambda}$ converges to $c_i^*$ in $\mathcal C^1(\mathbb R)$ as $\Lambda$ tends to infinity.
Then ${\bf c}^*=[c_0^*,\cdots,c_N^*]^{\mathsf{T}}$ satisfies
\[\frac{\mathrm{d}{\bf c}^*}{\mathrm{d}\phi}=-\lim_{\Lambda\to\infty}(D_\Lambda+\Lambda\tilde{G})^{-1}{\bf z}=\lim_{\Lambda\to\infty}\frac{\mathrm{d}{\bf c}_{\Lambda}}{\mathrm{d}\phi}=-\lim_{\Lambda\to\infty}\Lambda^{-1}\tilde{G}^{-1}(I+\Lambda^{-1}D_{\Lambda}\tilde{G}^{-1})^{-1}{\bf z}={\bf 0},
\]and hence, each $c_i^*$ becomes constant which is trivial.
Therefore, when $D_\Lambda$ is uniformly bounded for $\Lambda\geq1$, matrix $\tilde{G}$ must be singular in order to have nontrivial limiting functions $c_i^*=c_i^*(\phi)$ for $i=0,1,\dots,N$.

In the rest of this paper, we use the assumptions of steric effects given by
\begin{itemize}
\item[(A1)] $g_{ij}=\Lambda\lambda_i\lambda_j$ for $i,j=0,1,\dots,N$,
\item[(A2)] $\mu_i=\Lambda\lambda_i\tilde{\mu}_0+\hat{\mu}_i$ for $i=0,1,\dots,N$,
\end{itemize}
where $\lambda_i$, $\tilde{\mu}_0$ and $\hat{\mu}_i$ are positive constants independent of $\Lambda$.
Note that (A1) assures that $\tilde{G}=[\lambda_i\lambda_j]$ is a singular Gramian matrix which is rank one.
By (A1) and (A2), equation \eqref{eq:1.06} has unique solutions $c_{i,\Lambda}=c_{i,\Lambda}(\phi)$ satisfying
\begin{align}
\label{eq:1.10}
&c_{i,\Lambda}=(c_{0,\Lambda})^{\lambda_i/\lambda_0}\exp(\bar{\mu}_i-z_i\phi)\quad\text{for}~\phi\in\mathbb R~\text{and}~i=1,\dots,N,\\
\label{eq:1.11}&
\ln c_{0,\Lambda}+\Lambda\lambda_0\sum_{j=0}^N\lambda_j(c_{0,\Lambda})^{\lambda_j/\lambda_0}\exp(\bar{\mu}_j-z_j\phi)=\mu_0\quad\text{for}~\phi\in\mathbb R,
\end{align}
where $\bar{\mu}_i=\hat{\mu}_i-\frac{\lambda_i}{\lambda_0}\hat{\mu}_0$ is a constant independent of $\Lambda$ for $i=0,1,\dots,N$.
Then we apply the implicit function theorem on \eqref{eq:1.11} and obtain that $c_{0,\Lambda}(\phi)$ is a positive smooth function (see Proposition~\ref{prop:2.1}).
Hence equation \eqref{eq:1.07} becomes a nonlinear elliptic equation which is called a Poisson--Boltzmann equation with steric effects (PB-steric equation) and is expressed as
\begin{equation}
\label{eq:1.12}
-\nabla\cdot(\epsilon\nabla\phi_{\Lambda})=\rho_0+f_{\Lambda}(\phi_{\Lambda})\quad\text{in}~\Omega,
\end{equation}
where function $f_\Lambda=f_\Lambda(\phi)$ is denoted as
\begin{equation}
\label{eq:1.13}
f_\Lambda(\phi)=\sum_{i=1}^Nz_ic_{i,\Lambda}(\phi)\quad\text{for}~\phi\in\mathbb R.
\end{equation}
Notice that by \eqref{eq:1.10}, \eqref{eq:1.11} and \eqref{eq:1.13}, functions $f_\Lambda$ is the nonlinear term of PB-steric equation \eqref{eq:1.12}, and is mainly determined by the Lambert type function $c_{0,\Lambda}$ (given by \eqref{eq:1.11}) which represents the concentration of solvent molecules.
The main difficulty of this paper is to establish the convergence of $c_{i,\Lambda}$ and $\phi_\Lambda$ as $\Lambda$ tends to infinity.

By the implicit function theorem on Banach spaces (cf. \cite[Theorem 15.1]{1973deimling}), we prove that $c_{i,\Lambda}$ converges to $c_i^*$ in space $\mathcal{C}^m\left[a,b\right]$ as $\Lambda$ tends to infinity for $m\in\mathbb N$ and $a<b$ (see Propositions~\ref{prop:2.5} and \ref{prop:2.6}).
Here function $c_i^*$ satisfies
\begin{align}
\label{eq:1.14}
&c_i^*(\phi)=(c_0^*(\phi))^{\lambda_i/\lambda_0}\exp\left(\bar{\mu}_i-z_i\phi\right)>0\quad\text{for}~\phi\in\mathbb R~\text{and}~i=1,\dots,N,\\
\label{eq:1.15}
&\sum_{i=0}^N\lambda_ic_i^*\left(\phi\right)=\sum_{i=0}^N\lambda_i(c_0^*(\phi))^{\lambda_i/\lambda_0}\exp(\bar{\mu}_i-z_i\phi)=\tilde{\mu}_0\quad\text{for}~\phi\in\mathbb R,
\end{align}
where $\bar{\mu}_i=\hat{\mu}_i-\frac{\lambda_i}{\lambda_0}\hat{\mu}_0$.
Thus function $f_\Lambda$ also converges to $f^*$ in space $\mathcal{C}^m\left[a,b\right]$ as $\Lambda$ goes to infinity for $m\in\mathbb N$ and $a<b$, where
\begin{equation}
\label{eq:1.16}
f^*(\phi)=\sum_{i=1}^Nz_ic_i^*(\phi)
\end{equation}
(see Figure~\ref{fig1}). Note that for $\Lambda>0$, function $f_\Lambda$ is strictly decreasing and unbounded on the entire space $\mathbb R$ (see Propositions~\ref{prop:2.2} and \ref{prop:2.4}), and so does function $f_{PB}(\phi)=\sum_{i=1}^Nz_i\exp(-z_i\phi)$ of the conventional Poisson--Boltzmann (PB) equation $-\nabla\cdot(\epsilon\nabla\phi)=\rho_0+f_{PB}(\phi)$ in $\Omega$.
Besides, function $f^*$ is also strictly decreasing but bounded on $\mathbb R$ (see  Propositions~\ref{prop:2.7} and \ref{prop:2.8}), and so does function $f_{mPB}(\phi)$ of the mPB equation \eqref{eq:1.18} in \cite{1942bikerman,1997borukhov,2007kilic}.
Obviously, $f_\Lambda$ cannot uniformly converge to $f^*$ on $\mathbb R$ as $\Lambda$ goes to infinity.
Hence we only have uniform convergence of $f_\Lambda$ on any bounded interval $\left[a,b\right]$ but not the entire space $\mathbb R$.
To get the asymptotic limit of equation \eqref{eq:1.12}, we firstly have to prove the uniform boundedness of $\phi_\Lambda$ (the solution of equation \eqref{eq:1.12}) with respect to $\Lambda$ (see Lemmas~\ref{lma:3.2} and \ref{lma:4.1}) in order to use the convergence of $f_\Lambda$ in space $\mathcal{C}^m[a,b]$ for $m\in\mathbb N$ and $a<b$.
Here equation \eqref{eq:1.10} is crucial for the proof of Lemmas~\ref{lma:3.2} and \ref{lma:4.1}.
Notice that $\rho_0=\rho_0(x)$ may be any non-zero function and the boundary condition of equation \eqref{eq:1.12} may be the Robin or Neumann boundary condition but not the Dirichlet boundary condition so one cannot simply use the maximum principle on equation \eqref{eq:1.12} to prove Lemmas~\ref{lma:3.2} and \ref{lma:4.1}.
By Lemmas~\ref{lma:3.2} and \ref{lma:4.1}, we apply the $W^{2,p}$-estimate (cf. \cite[Theorem 15.2]{1959agmon}), the Schauder estimate (cf. \cite[Theorem 6.30]{1977gilbarg}) and the uniqueness of solution of equation \eqref{eq:1.17} to prove that $\phi_\Lambda$ converges to $\phi^*$ in space $\mathcal{C}^m(\Omega)$ as $\Lambda$ tends to infinity for $m\in\mathbb N$, where $\phi^*$ satisfies
\begin{equation}\label{eq:1.17}
-\nabla\cdot(\epsilon\nabla\phi^*)=\rho_0+f^*(\phi^*)\quad\text{in}~\Omega\end{equation}
with the Robin and Neumann boundary conditions \eqref{eq:1.08} and \eqref{eq:1.09}, respectively.
\begin{figure}
\centering\includegraphics[scale=0.68]{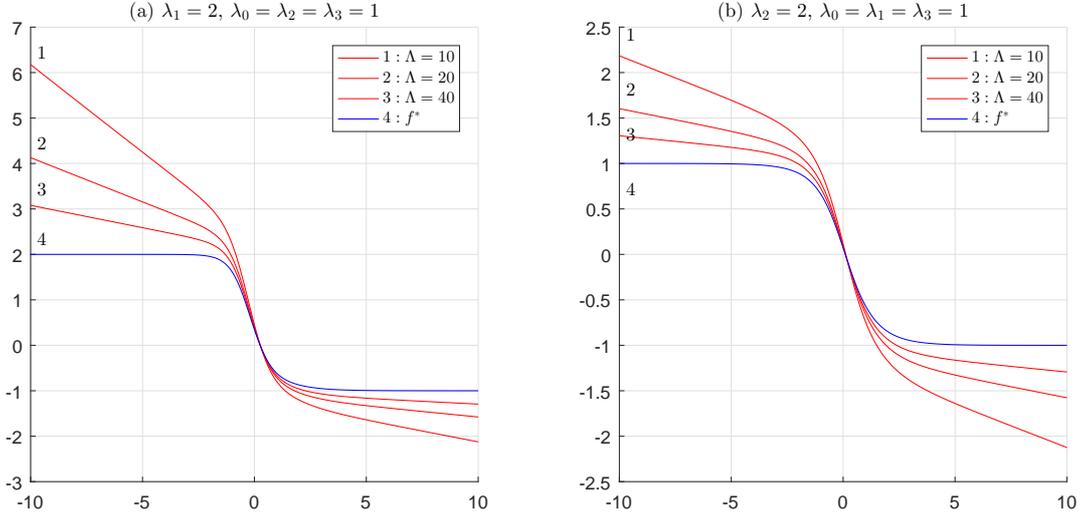}
\caption{Curves $1$--$3$ are profiles of function $f_\Lambda$ (defined in \eqref{eq:1.13}) with $\Lambda=10,20,40$, and curve $4$ is the profile of function $f^*$ (defined in \eqref{eq:1.16}).}
\label{fig1}
\end{figure}

Now we state the following theorems. 
\begin{thm}
\label{thm:1.1}
Let $\Omega\subsetneq\mathbb{R}^d$ be a bounded smooth domain, $\epsilon\in\mathcal{C}^{\infty}(\overline{\Omega})$ be a positive function, $\rho_0\in\mathcal{C}^{\infty}(\overline{\Omega})$, $\phi_{bd}\in\mathcal{C}^2(\partial\Omega)$, $z_0=0$, and $z_iz_j<0$ for some $i,j\in\{1,\dots,N\}$.
Assume that (A1) and (A2) hold true.
Then the PB-steric equation \eqref{eq:1.12} with the Robin boundary condition \eqref{eq:1.08} satisfies
\[\lim_{\Lambda\to\infty}\left\|\phi_\Lambda-\phi^*\right\|_{\mathcal{C}^m(\Omega)}=0\quad\text{for}~m\in\mathbb N,\]
where $\phi^*$ is the solution of equation \eqref{eq:1.17} with the Robin boundary condition \eqref{eq:1.08}.
\end{thm}
\begin{thm}
\label{thm:1.2}
Let $\Omega\subsetneq\mathbb{R}^d$ be a bounded smooth domain, $\epsilon\in\mathcal{C}^{\infty}(\overline{\Omega})$ be a positive function, $z_0=0$, and $z_iz_j<0$ for some $i,j\in\{1,\dots,N\}$.
Assume that (A1) and (A2) hold true. Then
\begin{enumerate}
\item[(i)] Suppose that $\rho_0$ is constant. Then the PB-steric equation \eqref{eq:1.12} with the Neumann boundary condition \eqref{eq:1.09} has a unique constant solution $\phi_\Lambda$.
Moreover if $\rho_0\in(-M^*,-m^*)$, then $\lim_{\Lambda\to\infty}\phi_\Lambda=\phi^*$, where $\phi^*$ is the unique constant solution of equation \eqref{eq:1.17} with the Neumann boundary condition \eqref{eq:1.09}, and
\[M^*=\lim_{\phi\to-\infty}f^*\left(\phi\right)>0,\quad m^*=\lim_{\phi\to\infty}f^*\left(\phi\right)<0.
\]
\item[(ii)] Suppose that $\rho_0\in\mathcal{C}^\infty(\overline{\Omega})$ is a nonconstant function with $\int_\Omega\rho_0\,\mathrm{d}x=0$.
Then the PB-steric equation \eqref{eq:1.12} with the Neumann boundary condition \eqref{eq:1.09} has a unique solution $\phi_\Lambda$ satisfying
\[
\lim_{\Lambda\to\infty}\|\phi_{\Lambda}-\phi^*\|_{\mathcal{C}^m(\Omega)}=0\quad\text{for}~m\in\mathbb N,
\]
where $\phi^*$ is the unique solution of equation \eqref{eq:1.17} with the Neumann boundary condition \eqref{eq:1.09}.
\end{enumerate}
\end{thm}
\begin{rmk}\label{rmk1}
As $\Lambda=0$, equations \eqref{eq:1.10} and \eqref{eq:1.11} imply $c_{0,0}=e^{\hat{\mu}_0}$ and $c_{i,0}=\exp(\hat{\mu}_i-z_i\phi)$ for $i=1,\dots,N$.
Here we replace $\phi_0$ by $\phi$ so the PB-steric equation \eqref{eq:1.12} becomes the conventional Poisson--Boltzmann (PB) equation $-\nabla\cdot(\epsilon\nabla\phi)=\rho_0+\sum_{i=1}^Nz_i\exp(\hat{\mu}_i-z_i\phi)$ in $\Omega$. However, as $\Lambda>0$, a larger $\Lambda$ makes the steric repulsion (between ions and solvent molecules) stronger, and eventually (as $\Lambda\to\infty)$ the limiting equation of PB-steric equation \eqref{eq:1.12} becomes equation \eqref{eq:1.17} which has the same form (up to scalar multiples) as the modified PB (mPB) equations in \cite{1942bikerman,1997borukhov,2007kilic,2009li,2009li2,2013li,2011lu} (see below).
Hence the PB-steric equation \eqref{eq:1.12} is a new PB equation (with the steric effects of ions and solvent molecules) which can be regarded as a generalized model of mPB equations.
\end{rmk}

To show equation \eqref{eq:1.17} with the same form (up to scalar multiples) as the mPB equations \eqref{eq:1.18} and \eqref{eq:1.21}, we use \eqref{eq:1.04} and (A1) with $\Lambda=S_\sigma$ to obtain $\lambda_i\lambda_j=\epsilon_{ij}(a_i+a_j)^{12}$ for $i,j=0,1,\dots,N$ which implies $\lambda_i=(\epsilon_{ii}(2a_i)^{12})^{1/2}$ and $\epsilon_{ij}=(\epsilon_{ii}\epsilon_{jj})^{1/2}\frac{2^{12}(a_ia_j)^6}{(a_i+a_j)^{12}}$ for $i,j=0,1,\dots,N$.
Note that because of $S_\sigma\sim\sigma^{d-12}$ and $1\leq d\leq3$, $S_\sigma\to\infty$ iff $\sigma\to0^+$, where $\sigma>0$ is the parameter of the approximate Lennard--Jones potential (cf. \cite{2014lin}).
Hence the nonlinear term of equation \eqref{eq:1.17} is function $f^*$ which depends on $a_i$'s the radii of ions and solvent molecules.
Suppose that $\lambda_i=a_i^3$, i.e., $\epsilon_{ii}=(2^{12}a_i^6)^{-1}$ for $i=0,1,\dots,N$.
If ions and solvent molecules have the same size, i.e., $a_i=a_0$ for all $i$, then equations \eqref{eq:1.14} and \eqref{eq:1.15} imply
\[c_i^*(\phi)=\frac{\tilde{\mu}_0}{a_0^3}\frac{\exp(\bar{\mu}_i-z_i\phi)}{\sum_{i=0}^N\exp(\bar{\mu}_i-z_i\phi)}\quad\text{for}~\phi\in\mathbb R~\text{and}~i=0,1,\dots,N,\]which can be plugged in equations \eqref{eq:1.16} and \eqref{eq:1.17} (with $\phi=\phi^*$) to get the following mPB equation (same size of ions and solvent molecules)
\begin{equation}
\label{eq:1.18}
-\nabla\cdot(\epsilon\nabla\phi)=\rho_0+f_{mPB}(\phi)\quad\text{in}~\Omega,
\end{equation}
where
\[f_{mPB}(\phi)=\frac{\tilde{\mu}_0}{a_0^3}\frac{\sum_{i=1}^Nz_i\exp(\bar{\mu}_i-z_i\phi)}{\sum_{i=0}^N\exp(\bar{\mu}_i-z_i\phi)}\quad\text{for}~\phi\in\mathbb R,\]see \cite{1942bikerman,1997borukhov,2007kilic}.
Conversely, if ions and solvent molecules do not have the same size, i.e., there exists $a_i\neq a_0$ for some $i\in\{1,\dots,N\}$, then equations \eqref{eq:1.14}--\eqref{eq:1.17} can be transformed to the following mPB equations
\begin{align}
\label{eq:1.19}
&\sum_{i=0}^Nv_ic_i=1,\\
\label{eq:1.20}
&\frac{v_i}{v_0}\ln(v_0c_0)-\ln(v_ic_i)=\beta q_i\psi-\beta\mu_i,\\
\label{eq:1.21}
&-\nabla\cdot(\epsilon\nabla\phi)=\rho_0+\sum_{i=0}^N\beta q_iv_ic_i(\psi)~~\text{in}~\Omega,
\end{align}
see \cite{2009li,2009li2,2013li,2011lu}.
Here we set $\lambda_i=a_i^3=v_i\tilde{\mu}_0$, $c_i^*=c_i$, $\bar{\mu}_i=\beta\mu_i+\frac{v_i}{v_0}\ln v_0-\ln v_i$, $z_i=\beta q_i$ and $\phi^*=\phi=\psi$.
Moreover, equations \eqref{eq:1.14}--\eqref{eq:1.17} have the same form (up to scalar multiples) as equations \eqref{eq:1.19}--\eqref{eq:1.21} (see Appendix II).
Therefore, as for Remark~\ref{rmk1}, the PB-steric equation (equation \eqref{eq:1.12} with the nonlinear term given by equations \eqref{eq:1.10}, \eqref{eq:1.11} and \eqref{eq:1.13}) can be regarded as a generalized model of mPB equations.

In Section~\ref{sec5}, we provide numerical simulations on equations \eqref{eq:1.12} and \eqref{eq:1.17} with the Robin and Neumann boundary conditions \eqref{eq:1.08} and \eqref{eq:1.09}, respectively.
The one-dimensional domain $\Omega=\left(-1,1\right)$ is discretized by the Legendre--Gauss--Lobatto (LGL) points (cf. \cite{2015elbaghdady}).
Then we use the LGL points to discretize equations \eqref{eq:1.12} and \eqref{eq:1.17}, and solve them numerically with the command {\ttfamily fsolve} in Matlab.
Under the Robin boundary condition \eqref{eq:1.08}, we show the profiles of solutions $\phi_\Lambda$ and $\phi^*$ to support Theorem~\ref{thm:1.1}.
Moreover, we demonstrate the profiles of solutions $\phi_\Lambda$ and $\phi^*$ under the Neumann boundary condition \eqref{eq:1.09} to support Theorem~\ref{thm:1.2}.

{\bf Organization.} The rest of the paper is organized as follows.
In Section~\ref{sec2}, the analysis of functions $f_\Lambda$ and $f^*$ is presented.
We prove Theorems~\ref{thm:1.1} and \ref{thm:1.2} in Sections~\ref{sec3} and \ref{sec4}, respectively.
Numerical results are shown in Section~\ref{sec5}.

\section{Analysis of \texorpdfstring{$f_\Lambda$}{fΛ} and \texorpdfstring{$f^*$}{f*}}
\label{sec2}

In this section, we show how to solve equation \eqref{eq:1.06} by (A1) and (A2) and obtain smooth solutions $c_{i,\Lambda}=c_{i,\Lambda}(\phi)$ for $\phi\in\mathbb R$ and $i=0,1,\dots,N$.
Note that because of (A1), equation \eqref{eq:1.06} can be expressed as \begin{equation}
\label{eq:2.01}
\ln c_{i,\Lambda}+z_i\phi+\Lambda\sum_{j=0}^N\lambda_i\lambda_jc_{j,\Lambda}=\mu_j\quad\text{for}~\phi\in\mathbb R~\text{and}~i=0,1,\dots,N.
\end{equation}
By (A2) and Gaussian elimination, we solve equation \eqref{eq:2.01} and get \eqref{eq:1.10}.
Then we plug \eqref{eq:1.10} into \eqref{eq:2.01} with $i=0$ and obtain equation \eqref{eq:1.11}.
Moreover, we apply the implicit function theorem to prove that equation \eqref{eq:1.11} has a smooth solution $c_{0,\Lambda}=c_{0,\Lambda}(\phi)$ for $\phi\in\mathbb R$ (see Proposition~\ref{prop:2.1}).
As $\Lambda$ goes to infinity, we use the implicit function theorem on Banach space to show that $c_{i,\Lambda}$ converges to $c_i^*$ in space $\mathcal{C}^m[a,b]$ for any $m\in\mathbb N$ and $a<b$ (see Proposition~\ref{prop:2.6}).
Hence $f_\Lambda=\sum_{i=1}^Nz_ic_{i,\Lambda}(\phi)$ also converges to $f^*=\sum_{i=1}^Nz_ic_i^*(\phi)$ in space $\mathcal{C}^m[a,b]$.

\subsection{Analysis of \texorpdfstring{$c_{i,\Lambda}$}{ciΛ} and \texorpdfstring{$f_\Lambda$}{fΛ}}
\label{sec2.1}

For functions $c_{i,\Lambda}$ and $f_\Lambda$, we have the following propositions.

\begin{prop}
\label{prop:2.1}
Assume that $z_0=0$, (A1) and (A2) hold true.
Then equation \eqref{eq:2.01} has smooth and positive solutions $c_{i,\Lambda}=c_{i,\Lambda}(\phi)$ for $\phi\in\mathbb R$ and $i=0,1,\cdots,N$.
\end{prop}
\begin{proof}
Multiplying the equation \eqref{eq:2.01} for $i=0$ by $\lambda_i/\lambda_0$, we obtain
\begin{equation}
\label{eq:2.02}
\frac{\lambda_i}{\lambda_0}\ln c_{0,\Lambda}+\Lambda\sum_{j=0}^N\lambda_i\lambda_jc_{j,\Lambda}=\frac{\lambda_i}{\lambda_0}\mu_0.
\end{equation}
Then we subtract the equation \eqref{eq:2.01} from the equation \eqref{eq:2.02} to get
\[\ln c_{i,\Lambda}-\frac{\lambda_i}{\lambda_0}\ln c_{0,\Lambda}+z_i\phi=\mu_i-\frac{\lambda_i}{\lambda_0}\mu_0,
\]
which implies
\begin{equation}
\label{eq:2.03}
c_{i,\Lambda}=c_{0,\Lambda}^{\lambda_i/\lambda_0}\exp\left(\bar{\mu}_i-z_i\phi\right)\quad\text{for}~\phi\in\mathbb R~\text{and}~i=1,\dots,N,
\end{equation}
where $\bar{\mu_i}=\mu_i-\frac{\lambda_i}{\lambda_0}\mu_0=\hat{\mu}_i-\frac{\lambda_i}{\lambda_0}\hat{\mu}_0$ because of (A2).
Note that $\bar{\mu}_i$ is independent of $\Lambda$.
Plugging \eqref{eq:2.03} into \eqref{eq:2.02}, we have 
\begin{equation}
\label{eq:2.04}
\ln c_{0,\Lambda}+\Lambda\lambda_0\sum_{j=0}^N\,\lambda_j\,c_{0,\Lambda}^{\lambda_j/\lambda_0}\,\exp(\bar{\mu}_j-z_j\phi)-\mu_0=0\quad\text{for}~\phi\in\mathbb R,
\end{equation}
which can be denoted as $
h_1(c_{0,\Lambda},\phi)=0$.
Here $h_1$ is defined by
\[h_1\left(t,\phi\right)=\ln t+\Lambda\lambda_0\sum_{j=0}^N\lambda_jt^{\lambda_j/\lambda_0}\exp\left(\bar{\mu}_j-z_j\phi\right)-\mu_0\quad\text{for}~t>0~\text{and}~\phi\in\mathbb R.
\]
Notice that, for any $\phi\in\mathbb{R}$, $h_1$ is strictly increasing for $t>0$ and the range of $h_1$ is entire space $\mathbb{R}$. Then there exists a unique positive number $c_{0,\Lambda}(\phi)$ such that $h_1\left(c_{0,\Lambda}(\phi),\phi\right)=0$ for $\phi\in\mathbb R$.
Moreover, since $h_1$ is smooth for $t>0$, $\phi\in\mathbb R$, and
\[
\frac{\partial h_1}{\partial t}\left(t,\phi\right)=\frac1t+\Lambda\sum_{j=0}^N\lambda_j^2t^{(\lambda_j-\lambda_0)/\lambda_0}\exp\left(\bar{\mu}_j-z_j\phi\right)>0\quad\text{for}~t>0~\text{and}~\phi\in\mathbb R,
\]
then by the implicit function theorem (cf. \cite[Theorem 3.3.1]{2002krantz}), $c_{0,\Lambda}=c_{0,\Lambda}(\phi)$ is a smooth and positive function on $\mathbb R$.
Therefore, by \eqref{eq:2.03}, each $c_{i,\Lambda}$ is smooth and positive on $\mathbb R$ and we complete the proof of Proposition~\ref{prop:2.1}.
\end{proof}

\begin{prop}
\label{prop:2.2}
Function $\displaystyle f_\Lambda(\phi)=\sum_{i=0}^Nz_ic_{i,\Lambda}(\phi)$ is strictly decreasing on $\mathbb R$, where functions $c_{i,\Lambda}$ are obtained in Proposition~\ref{prop:2.1}.
\end{prop}
\begin{proof}
By Proposition~\ref{prop:2.1}, we differentiate the equation \eqref{eq:2.01} with respect to $\phi$ and obtain
\begin{equation}
\label{eq:2.05}
(D_\Lambda+G)\frac{\mathrm{d}{\bf c}_\Lambda}{\mathrm{d}\phi}=-{\bf z}\quad\text{for}~\phi\in\mathbb R,
\end{equation}
where matrix $D_\Lambda=\text{diag}(1/c_{0,\Lambda},\cdots,1/c_{N,\Lambda})$ is positive definite, matrix $G=\Lambda[\lambda_0~\cdots~\lambda_N]^{\mathsf{T}}[\lambda_0~\cdots~\lambda_N]$, ${\bf c}_\Lambda=[c_{0,\Lambda}~\cdots~c_{N,\Lambda}]^{\mathsf{T}}$ and ${\bf z}=[z_0~\cdots~z_N]^{\mathsf{T}}$.
It is obvious that matrix $D_\Lambda+G$ is positive definite and invertible with inverse matrix $(D_\Lambda+G)^{-1}$ which is also positive definite. Then equation \eqref{eq:2.05} gives $\mathrm{d}{\bf c}_\Lambda/\mathrm{d}\phi=-(D_\Lambda+G)^{-1}{\bf z}$, and $z^{\mathsf{T}}(\mathrm{d}{\bf c}_\Lambda/\mathrm{d}\phi)$ becomes 
\[
\frac{\mathrm{d}f_{\Lambda}}{\mathrm{d}\phi}=\sum_{i=1}^Nz_i\frac{\mathrm{d}c_{i,\Lambda}}{\mathrm{d}\phi}={\bf z}^{\mathsf{T}}\frac{\mathrm{d}{\bf c}}{\mathrm{d}\phi}=-{\bf z}^{\mathsf{T}}\left(D+G\right)^{-1}{\bf z}<0\quad\text{for}~\phi\in\mathbb R.
\]
Here we have used the fact that ${\bf z}\neq0$ and we complete the proof of Proposition~\ref{prop:2.2}.
\end{proof}

\begin{prop}
\label{prop:2.3}
Assume that (A1) and (A2) hold true.
\begin{enumerate}
\item[(i)] If $z_k\geq0$ ($z_k\leq0$) and $k\in\{0,1,\dots,N\}$, then $\displaystyle\sup_{\phi\geq0}c_{k,\Lambda}(\phi)\leq e^{\mu_k}$ ($\displaystyle\sup_{\phi\leq0}c_{k,\Lambda}(\phi)\leq  e^{\mu_k}$).
\item[(ii)] Suppose $z_iz_j<0$ for some $i,j\in\{1,\dots,N\}$.
Then there exist $i_0,j_0\in\{1,\dots,N\}$, $z_{i_0}>0$ and $z_{j_0}<0$ such that $\displaystyle\limsup_{\phi\to-\infty}c_{i_0,\Lambda}(\phi)=\infty$ and $\displaystyle\limsup_{\phi\to\infty}c_{j_0,\Lambda}(\phi)=\infty$.
\end{enumerate}
\end{prop}
\begin{proof}
Suppose that $z_k\geq0$ for some $k\in\{0,\dots,N\}$.
Then equation \eqref{eq:2.01} implies
\[\sup_{\phi\geq0}\ln c_{k,\Lambda}(\phi)=\sup_{\phi\geq0}\left(\mu_k-z_k\phi-\sum_{j=0}^Ng_{kj}c_{j,\Lambda}(\phi)\right)\leq\mu_k<\infty,\]
and $\displaystyle\sup_{\phi\geq0}c_{k,\Lambda}(\phi)\leq e^{\mu_k}$.
Here we have used the fact that $c_{i,\Lambda}(\phi)>0$ for $\phi\in\mathbb R$ and $i=0,1,\cdots,N$.
Similarly, if $z_k\leq0$ and $k\in\{0,\cdots,N\}$, then equation \eqref{eq:2.01} gives\[\sup_{\phi\leq0}\ln c_{k,\Lambda}(\phi)=\sup_{\phi\leq0}\left(\mu_k-z_k\phi-\Lambda\lambda_k\sum_{j=0}^N\lambda_jc_{j,\Lambda}(\phi)\right)\leq\mu_k<\infty,\]and $\displaystyle\sup_{\phi\leq 0}c_{k,\Lambda}(\phi)\leq e^{\mu_k}$.
Hence the proof of Proposition~\ref{prop:2.3}(i) is complete.

It remains to prove (ii).
Since $z_iz_j<0$ for some $i,j\in\{1,\dots,N\}$, then box index sets $I=\{i:z_i>0\}$ and $J=\{j:z_j<0\}$ are nonempty.
Now we claim that there exists $i_0\in I$ such that $\displaystyle\limsup_{\phi\to-\infty}c_{i_0,\Lambda}(\phi)=\infty$.
We prove this by contradiction.
Suppose $\displaystyle\sup_{\phi\leq0}c_{i,\Lambda}(\phi)<\infty$ for all $i\in I$.
Then there exists $K_1>0$ such that $0<c_{i,\Lambda}(\phi)<K_1$ for $\phi\leq0$ and $i\in I$.
By equation \eqref{eq:2.01} and Proposition \eqref{eq:2.03}(i), we have
\[\begin{aligned}z_i\phi&=\mu_i-\ln c_{i,\Lambda}(\phi)-\Lambda\lambda_i\sum_{k=0}^N\lambda_jc_{k,\Lambda}(\phi)\\&=\mu_i-\ln c_{i,\Lambda}(\phi)-\Lambda\lambda_i\left(\sum_{k\in I}\lambda_kc_{k,\Lambda}(\phi)+\sum_{j\in\{0\}\cup J}\lambda_kc_{k,\Lambda}(\phi)\right)\\&\geq\mu_i-\ln K_1-\Lambda\lambda_i\left(\sum_{k\in I}\lambda_kK_1+\sum_{k\in\{0\}\cup J}\lambda_ke^{\mu_k}\right)\quad\text{for}~i\in I~\text{and}~\phi\leq0,\end{aligned}\]
which leads a contradiction by letting $\phi\to-\infty$.
Hence there exists $i_0\in I$ such that $\displaystyle\limsup_{\phi\to-\infty}c_{i_0,\Lambda}(\phi)=\infty$.
Similarly, we claim that there exists $j_0\in J$ such that $\displaystyle\limsup_{\phi\to\infty}c_{j_0,\Lambda}(\phi)=\infty$.
We also prove this by contradiction.
Suppose $\displaystyle\sup_{\phi\geq0}c_{j,\Lambda}(\phi)<\infty$ for all $j\in J$.
Then there exists $K_2>0$ such that $0<c_{j,\Lambda}(\phi)\leq K_2$ for $\phi\geq0$ and $j\in J$.
By equation \eqref{eq:2.01} and Proposition~\ref{prop:2.3}(i), we have
\[\begin{aligned}z_j\phi&=\mu_j-\ln c_{j,\Lambda}(\phi)-\Lambda\lambda_j\sum_{k=0}^N\lambda_kc_{k,\Lambda}(\phi)\\&=\mu_j-\ln c_{j,\Lambda}(\phi)-\Lambda\lambda_j\left(\sum_{k\in\{0\}\cup I}\lambda_kc_{k,\Lambda}(\phi)+\sum_{k\in J}\lambda_kc_{k,\Lambda}(\phi)\right)\\&\geq\mu_j-\ln K_2-\Lambda\lambda_j\left(\sum_{k\in\{0\}\cup I}\lambda_ke^{\mu_k}+\sum_{k\in J}\lambda_kK_2\right)\quad\text{for}~k\in J~\text{and}~\phi\geq0,\end{aligned}\]
which leads a contradiction by letting $\phi\to\infty$. Therefore, the proof of Proposition~\ref{prop:2.3}(ii) is complete.\end{proof}

\begin{prop}
\label{prop:2.4}
For each $\Lambda>0$, the range of $f_{\Lambda}$ is entire space $\mathbb{R}$, and $\displaystyle\lim_{\phi\to\pm\infty}f_\Lambda\left(\phi\right)=\mp\infty$.
\end{prop}
\begin{proof}By \eqref{eq:1.13}, function $f_\Lambda$ can be denoted as
\begin{equation}
\label{eq:2.06}f_{\Lambda}(\phi)=\sum_{z_i>0}z_ic_{i,\Lambda}(\phi)+\sum_{z_j<0}z_jc_{j,\Lambda}(\phi)\quad\text{for}~\phi\in\mathbb R.\end{equation}
Proposition~\ref{prop:2.3}(i) gives
\begin{equation}
\label{eq:2.07}\sum_{z_i>0}z_ic_{i,\Lambda}(\phi)<\sum_{z_i>0}z_ie^{\mu_i}\quad\text{for}~\phi\geq0.\end{equation}
Moreover, by Proposition~\ref{prop:2.3}(ii), there exists $j_0\in\{1,\dots,N\}$ with $z_{j_0}<0$ and a sequence $\{\phi_n\}_{n=1}^{\infty}$ with $\displaystyle\lim_{n\to\infty}\phi_n=\infty$ such that $\displaystyle\lim_{n\to\infty}z_{j_0}c_{j_0,\Lambda}(\phi_n)=-\infty$.
Thus, by \eqref{eq:2.06} and \eqref{eq:2.07}, we have\[f_{\Lambda}(\phi_n)\leq\sum_{z_i>0}z_ie^{\mu_i}+z_{j_0}c_{j_0,\Lambda}(\phi_n)\to-\infty\quad\text{as}~n\to\infty,\]which implies that $\displaystyle\lim_{\phi\to\infty}f_\Lambda(\phi)=-\infty$ because of the monotone decreasing of $f_\Lambda$ (see Proposition~\ref{prop:2.2}).
On the other hand, Proposition~\ref{prop:2.3}(i), also implies
\begin{equation}
\label{eq:2.08}
\sum_{z_i<0}z_ic_{i,\Lambda}(\phi)\geq\sum_{z_i<0}z_ie^{\mu_i}\quad\text{for}~\phi\leq0.
\end{equation}
Moreover, by Proposition~\ref{prop:2.3}(ii), there exist $i_0\in\{1,\dots,N\}$ with $z_{i_0}>0$ and a sequence $\{\tilde{\phi}_n\}_{n=1}^{\infty}$ with $\displaystyle\lim_{n\to\infty}\tilde{\phi}_n=-\infty$ such that $\displaystyle\lim_{n\to\infty}z_{i_0}c_{i_0,\Lambda}(\tilde{\phi}_n)=\infty$.
Thus, by \eqref{eq:2.06} and \eqref{eq:2.08}, we obtain
\[f_\Lambda(\tilde{\phi}_n)\geq z_{i_0}c_{i_0,\Lambda}(\tilde{\phi}_n)+\sum_{z_i<0}z_ie^{\mu_i}\to\infty\quad\text{as}~n\to\infty,\]
which implies that $\displaystyle\lim_{\phi\to-\infty}f_\Lambda(\phi)=\infty$ due to the monotone decreasing of $f_\Lambda$.
Therefore, the proof of Proposition~\ref{prop:2.4} is complete.
\end{proof}

\subsection{Analysis of \texorpdfstring{$c_i^*$}{ci*} and \texorpdfstring{$f^*$}{f*}}
\label{sec2.2}

Function $c_0^*$ is the limit $\displaystyle\lim_{\Lambda\to\infty}\,c_{0,\Lambda}$ (see Proposition~\ref{prop:2.6}), where function $c_{0,\Lambda}$ is the solution of equation \eqref{eq:2.04} for $\Lambda>0$.
Let $\delta=\Lambda^{-1}$ and $\tilde{c}_{0,\delta}=c_{0,\Lambda}$.
Then by (A2) and equation \eqref{eq:2.04}, function $\tilde{c}_{0,\delta}$ satisfies
\begin{equation}
\label{eq:2.09}
\delta\ln\tilde{c}_{0,\delta}(\phi)+\lambda_0\sum_{i=0}^N\,\lambda_i\,(\tilde{c}_{0,\delta}(\phi))^{\lambda_i/\lambda_0}\,\exp(\bar{\mu}_i-z_i\phi)=\lambda_0\tilde{\mu}_0+\delta\hat{\mu_0}\quad\text{for}~\phi\in\mathbb R.
\end{equation}
Notice that $\Lambda\to\infty$ is equivalent to $\delta\to0^+$ so $c_0^*$ is also equal to the limit $\displaystyle\lim_{\delta\to0^+}\,\tilde{c}_{0,\delta}$.
Moreover, $c_0^*$ satisfies equation \eqref{eq:1.15} which is equation \eqref{eq:2.09} with $\delta=0$.
Besides, by \eqref{eq:2.03}, functions $c_i^*$ satisfy the equation \eqref{eq:1.14} for $i=1,\dots,N$.
The existence and uniqueness of equations \eqref{eq:1.14} and \eqref{eq:1.15} is proved in Proposition~\ref{prop:2.5}.
The convergence of $\tilde{c}_{0,\delta}$ as $\delta\to0^+$, i.e. the convergence of $c_{0,\Lambda}$ as $\Lambda\to\infty$ is proved in Proposition~\ref{prop:2.6} so by \eqref{eq:2.03}, we obtain the convergence of $c_{i,\Lambda}$ as $\Lambda\to\infty$.

Now we state Propositions~\ref{prop:2.5} and \ref{prop:2.6} as follows.    
\begin{prop}
\label{prop:2.5}
Equations \eqref{eq:1.14} and \eqref{eq:1.15} have a unique solution $(c_0^*,\dots,c_N^*)$ and each function $c_i^*=c_i^*(\phi)$ is a smooth and positive function for $\phi\in\mathbb R$ and $i=0,1,\dots,N$.
\end{prop}
\begin{proof}
Equations \eqref{eq:1.14} and \eqref{eq:1.15} can be solved by the following problem.
\[\sum_{j=0}^N\,\lambda_j(c_0^*(\phi))^{\lambda_j/\lambda_0}\,\exp(\bar{\mu}_j-z_j\phi)-\tilde{\mu}_0=0\quad\text{for}~\phi\in\mathbb R,\]which can be represented by $h_2(c_0^*(\phi),\phi)=0$ for $\phi\in\mathbb R$.
Here $h_2$ is defined by
\[h_2(t,\phi)=\sum_{j=0}^N\lambda_jt^{\lambda_j/\lambda_0}\exp(\bar{\mu}_j-z_j\phi)-\tilde{\mu}_0\quad\text{for}~t>0~\text{and}~\phi\in\mathbb R.
\]
Note that $\displaystyle\lim_{t\to0^+}h_2(t,\phi)=-\tilde{\mu}_0<0$ and $\displaystyle\lim_{t\to\infty}h_2(t,\phi)=\infty$ for any $\phi\in\mathbb R$. Thus, for any $\phi\in\mathbb R$, there exists $t>0$ such that $h_2(t,\phi)=0$. Since $\frac{\partial h_2}{\partial t}=\frac1{\lambda_0}\sum_{j=0}^N\lambda_j^2t^{(\lambda_j-\lambda_0)/\lambda_0}\exp(\bar{\mu}_j-z_j\phi)>0$ for $t>0$ and $\phi\in\mathbb R$, then we may use the implicit function theorem (cf. \cite[Theorem~3.3.1]{2002krantz}) to conclude that there exists a unique smooth and positive function $c_0^*:\mathbb R\to\mathbb R^+$ such that $h_2(c_0^*(\phi),\phi)=0$ for $\phi\in\mathbb R$. Therefore by equation \eqref{eq:1.14}, we obtain the smooth and positive functions $c_i^*$ for $i=1,\dots,N$, and complete the proof of Proposition~\ref{prop:2.5}.
\end{proof}
\begin{prop}
\label{prop:2.6}
$\displaystyle\lim_{\Lambda\to\infty}\|c_{i,\Lambda}-c_i^*\|_{\mathcal{C}^m\left[a,b\right]}=0$ for $i=0,1,\dots,N$, $a<b$ and $m\in\mathbb{N}$, where $\displaystyle\|h\|_{\mathcal{C}^m\left[a,b\right]}:=\sum_{k=0}^m\|h^{(k)}\|_{\infty}$ for $h\in\mathcal C^m\left[a,b\right]$.
\end{prop}
\begin{proof}Fix $m\in\mathbb N$, $a,b\in\mathbb R$ and $a<b$ arbitrarily. Let $\|\cdot\|_{\mathcal{C}^m}:=\|\cdot\|_{\mathcal{C}^m[a,b]}$ for notation convenience.
For $\Lambda>0$, let $\delta=\Lambda^{-1}$, $\tilde{c}_{0,\delta}=c_{0,\Lambda}$, $w_\delta=\ln\tilde{c}_{0,\delta}$ and $w^*=\ln c_0^*$.
Obviously, $\delta\to0^+$ is equivalent to $\Lambda\to\infty$.
Hence by \eqref{eq:2.03}, it suffices to show that $\displaystyle\lim_{\delta\to0^+}\|e^{w_\delta}-e^{w^*}\|_{\mathcal{C}^m}=0$.
Because $w_\delta=\ln\tilde{c}_{0,\delta}$ and $\tilde{c}_{0,\delta}=c_{0,\Lambda}$, equation \eqref{eq:2.09} can be denoted as
\begin{equation}
\label{eq:2.10}
H(w_\delta(\phi),\delta)=0
\end{equation}
for $\delta>0$ and $\phi\in [a,b]$, where $H$ is a $\mathcal{C}^1$-function on $\mathcal{C}^m[a,b]\times\mathbb R$ defined by
\begin{equation}
\label{eq:2.11}
H(w(\phi),\delta)=-\tilde{\mu}_0+\sum_{i=0}^N\lambda_i\exp\left(\frac{\lambda_i}{\lambda_0}w(\phi)+\bar{\mu}_i-z_i\phi\right)-\frac{\hat{\mu}_0-w(\phi)}{\lambda_0}\cdot\delta
\end{equation}
for all $w\in\mathcal{C}^m[a,b]$ and $\phi\in[a,b]$.
Note that $H(w^*,0)=0$ by Proposition~\ref{prop:2.5}.
A direct calculation for Fr\'{e}chet derivative of \eqref{eq:2.11} gives
\[D_wH(w(\phi),\delta)=\frac1{\lambda_0}\sum_{i=0}^N\lambda_i^2\exp\left(\frac{\lambda_i}{\lambda_0}w(\phi)+\bar{\mu}_i-z_i\phi\right)+\frac{\delta}{\lambda_0}\]
for all $w\in\mathcal{C}^m[a,b]$ and $\phi\in[a,b]$.
Then $D_wH(w^*(\phi),0)=\frac1{\lambda_0}\sum_{i=0}^N\lambda_i^2(c_0^*(\phi))^{\lambda_i/\lambda_0}\exp(\bar{\mu}_i-z_i\phi)>0$ for $\phi\in[a,b]$.
Here we have used the fact that $w^*=\ln c_0^*$.
This implies that $D_wH(w^*,0)I$ is a bounded and invertible linear map on the Banach space $\mathcal{C}^m[a,b]$, where $I$ is an identity map.
Hence by the implicit function theorem on Banach spaces (cf. \cite[Corollary 15.1]{1973deimling}), there exist an open subset $B_{\delta_0}(w^*)\times(-\delta_0,\delta_0)\subsetneq\mathcal{C}^m[a,b]\times\mathbb{R}$ and a unique $\mathcal C^1$-function $\tilde{w}(\cdot,\delta)$ of $\delta\in(-\delta_0,\delta_0)$ with $\tilde{w}(\cdot,\delta)\in B_{\delta_0}(w^*)\subsetneq\mathcal{C}^m[a,b]$ for $\delta\in(-\delta_0,\delta_0)$ such that 
\[H(\tilde{w}(\cdot,\delta),\delta)=0~~\text{for all}~\delta\in(-\delta_0,\delta_0),\]
which gives $\displaystyle\lim_{\delta\to0^+}\|\tilde{w}(\cdot,\delta)-w^*\|_{\mathcal{C}^m}=0$.
By Proposition~\ref{prop:2.1}, equation \eqref{eq:2.10} has a unique solution $w_\delta$, which implies $\tilde{w}(\cdot,\delta)=w_\delta(\cdot)$ for $\delta\in(-\delta_0,\delta_0)$.
Therefore, we obtain $\displaystyle\lim_{\delta\to0^+}\left\|w_\delta-w^*\right\|_{\mathcal{C}^m}=0$, i.e., $\displaystyle\lim_{\delta\to0^+}\|e^{w_\delta}-e^{w^*}\|_{\mathcal{C}^m}=0$, and complete the proof of Proposition~\ref{prop:2.6}.
\end{proof}

\begin{rmk}
\label{rmk2}
Because $\displaystyle f_\Lambda(\phi)=\sum_{i=1}^Nz_ic_{i,\Lambda}(\phi)$ and $\displaystyle f^*(\phi)=\sum_{i=1}^Nz_ic_i^*(\phi)$ for $\phi\in\mathbb R$, Proposition~\ref{prop:2.6} gives $\displaystyle\lim_{\Lambda\to\infty}\|f_\Lambda-f^*\|_{\mathcal{C}^m[a,b]}$ for $m\in\mathbb N$ and $a<b$.
\end{rmk}

For function $f^*$, we have the following propositions.
\begin{prop}
\label{prop:2.7}
Function $\displaystyle f^*(\phi)=\sum_{i=1}^Nz_ic_i^*(\phi)$ is strictly decreasing on $\mathbb R$, where functions $c_i^*$ are obtained in Proposition~\ref{prop:2.5}.
\end{prop}
\begin{proof}Differentiating \eqref{eq:1.14} and \eqref{eq:1.15} with respect to $\phi$ gives
\begin{align}
\label{eq:2.12}
&\sum_{j=0}^N\lambda_j\frac{\mathrm{d}c_j^*}{\mathrm{d}\phi}=0,\\
&\label{eq:2.13}\frac{\mathrm{d}c_i^*}{\mathrm{d}\phi}=\frac{\lambda_i}{\lambda_0}\frac{c_i^*}{c_0^*}\frac{\mathrm{d}c_0^*}{\mathrm{d}\phi}-z_ic_i^*,\quad\text{for}~\phi\in\mathbb R~\text{and}~i=1,\dots,N.\end{align}
Multiply the equation \eqref{eq:2.13} by $\lambda_i$ for $i=1,\dots,N$ and add them together.
Then by equation \eqref{eq:2.12}, we obtain
\[
\frac{\mathrm{d}c_0^*}{\mathrm{d}\phi}=\lambda_0c_0^*\frac{\sum_{i=1}^Nz_i\lambda_ic_i^*}{\sum_{i=0}^N\lambda_i^2c_i^*}\quad\text{for}~\phi\in\mathbb R.
\]
Consequently, we have
\[
\frac{\mathrm{d}f^*}{\mathrm{d}\phi}=\sum_{i=1}^Nz_i\frac{\mathrm{d}c_i^*}{\mathrm{d}\phi}
=\frac1{\lambda_0c_0^*}\frac{\mathrm{d}c_0^*}{\mathrm{d}\phi}\sum_{i=1}^Nz_i\lambda_ic_i^*-\sum_{i=1}^Nz_i^2c_i^*=\frac{(\sum_{i=1}^Nz_i\lambda_ic_i^*)^2}{\sum_{i=0}^N\lambda_i^2c_i^*}-\sum_{i=1}^Nz_i^2c_i^*<0\quad\text{for}~\phi\in\mathbb R.\]
Here the last inequality comes from Cauchy's inequality, and the fact that $\lambda_i>0$, $c_i^*>0$ for $i\in\{0,1,\dots,N\}$, and $z_iz_j<0$ for some $i,j\in\{1,\dots,N\}$.
Therefore, we complete the proof of Proposition~\ref{prop:2.7}.
\end{proof}

\begin{prop}
\label{prop:2.8}
Function $f^*$ is bounded and $m^*<f^*(\phi)<M^*$ for all $\phi\in\mathbb{R}$, where $\displaystyle m^*=\lim_{\phi\to\infty}f^*(\phi)<0$ and $\displaystyle M^*=\lim_{\phi\to-\infty}f^*(\phi)>0$.
\end{prop}
\begin{proof}To prove boundedness of $f^*$, we use Proposition~\ref{prop:2.5} and equation~\ref{eq:1.15} to get
\begin{equation}
\label{eq:2.14}
0<c_i^*(\phi)<\tilde{\mu}_0/\lambda_i\quad\text{for}~\phi\in\mathbb R~\text{and}~i=0,1,\dots,N,
\end{equation}
and
\begin{equation}
\label{eq:2.15}
|f^*(\phi)|\leq\sum_{i=1}^N|z_i|c_i^*(\phi)\leq\tilde{\mu}_0\sum_{i=1}^N\frac{|z_i|}{\lambda_i}\leq\tilde{\mu}_0N\max_{0\leq i\leq N}\frac{|z_i|}{\lambda_i}\quad\text{for}~\phi\in\mathbb R.
\end{equation}
On the other hand, by \eqref{eq:1.14} and \eqref{eq:1.16}, function $f^*$ can be represented as
\begin{equation}
\label{eq:2.16}
f^*(\phi)=\sum_{i=1}^Nz_i(c_0^*(\phi))^{\lambda_i/\lambda_0}\exp(\bar{\mu}_i-z_i\phi)\quad\text{for}~\phi\in\mathbb R.
\end{equation}
By \eqref{eq:2.15} and Proposition~\ref{prop:2.7}, function $f^*$ is strictly decreasing and bounded on $\mathbb R$ so the limit $\displaystyle\lim_{\phi\to\infty}f^*(\phi)$, denoted by $m^*$, exists and is finite.
By \eqref{eq:1.14} and \eqref{eq:2.14}, we obtain
\begin{equation}
\label{eq:2.17}
\lim_{\phi\to\infty}c_0^*(\phi)=\lim_{\phi\to\infty}\left[c_i^*(\phi)\exp(-\bar{\mu}_i+z_i\phi)\right]^{\lambda_0/\lambda_i}=0\quad\text{for}~z_i<0.
\end{equation}
Moreover, we use \eqref{eq:1.14} and \eqref{eq:2.17} to get
\begin{equation}
\label{eq:2.18}
\lim_{\phi\to\infty}c_i^*(\phi)=\lim_{\phi\to\infty}\left[(c_0^*(\phi))^{\lambda_i/\lambda_0}\exp(\bar{\mu}_i-z_i\phi)\right]=0\quad\text{for}~z_i>0,
\end{equation}
and hence \eqref{eq:2.16} and \eqref{eq:2.18} imply $\displaystyle m^*=\lim_{\phi\to\infty}f^*(\phi)\leq0$.
Now, we prove $m^*<0$ by contradiction.
Suppose $m^*=0$.
Then \eqref{eq:2.18} gives $\displaystyle\lim_{\phi\to\infty}c_i^*(\phi)=0$ for $i=0,1,\dots,N$ which contradicts with equation \eqref{eq:1.15} (by letting $\phi\to\infty$) and $\tilde{\mu}_0>0$.
Similarly, we may prove $\displaystyle\lim_{\phi\to-\infty}f^*(\phi)=M^*>0$ and complete the proof of Proposition~\ref{prop:2.8}.
\end{proof}

\section{Proof of Theorem~\ref{thm:1.1}}
\label{sec3}

The existence and uniqueness of $\phi_\Lambda$ can be proved by the standard Direct method. One may refer the proof in Appendix I.
The uniform boundedness of $\phi_\Lambda$ and the convergence of $\phi_\Lambda$ are proved as follows.

\subsection{Uniform boundedness of \texorpdfstring{$\phi_\Lambda$}{} and \texorpdfstring{$c_{0,\Lambda}(\phi_\Lambda)$}{}}
\label{sec3.1}

\begin{lma}
\label{lma:3.1}
There exists a constant $M_0\geq1$ independent of $\Lambda$ such that $\displaystyle\max_{x\in\overline{\Omega}}c_{0,\Lambda}(\phi_\Lambda(x))\leq M_0$ for $\Lambda\geq1$.
\end{lma}
\begin{proof}
Let $\displaystyle M_0=\max\left\{\frac{\tilde{\mu_0}}{\lambda_0}+\frac{\hat{\mu_0}}{\lambda_0^2}, 1\right\}\geq1$ and $\Lambda\geq1$.
We claim that $c_{0,\Lambda}(\phi_\Lambda(x))\leq M_0$ for $x\in\overline{\Omega}$.
Suppose that $\Omega_1=\{x\in\overline{\Omega}:c_{0,\Lambda}(\phi_\Lambda(x))>1\}$ is nonempty.
Otherwise, due to $M_0\geq1$, $c_{0,\Lambda}(\phi_\Lambda(x))\leq1\leq M_0$ for $x\in\overline{\Omega}$ which is trivial.
By \eqref{eq:2.04} and (A2), we obtain
\[c_{0,\Lambda}(\phi_{\Lambda}(x))\leq\frac{\tilde{\mu}_0}{\lambda_0}+\frac{\hat\mu_0}{\Lambda\lambda_0^2}\leq M_0\]
for $x\in\Omega_1$.
Here we have used the fact that $\ln c_{0,\Lambda}(\phi_\Lambda(x))>0$ for $x\in\Omega_1$, $\Lambda\geq1$, and $z_0=\bar\mu_0=0$.
Therefore, we complete the proof of Lemma~\ref{lma:3.1}.
\end{proof}

\begin{lma}
\label{lma:3.2}
There exists a positive constant $M_1$ independent of $\Lambda$ such that $\left\|\phi_{\Lambda}\right\|_{L^\infty(\Omega)}\leq M_1$ for  $\Lambda\geq1$.
\end{lma}
\begin{proof}
Let $\psi$ be the solution of the equation $-\nabla\cdot(\epsilon\nabla\psi)=\rho_0$ in $\Omega$ with the Robin boundary condition $\psi+\eta\frac{\partial\psi}{\partial\nu}=0$ on $\partial\Omega$, and let $\bar{\phi}_\Lambda=\phi_\Lambda-\psi$.
Then function $\bar{\phi}_\Lambda$ satisfies
\begin{equation}
\label{eq:3.01}
\nabla\cdot(\epsilon\nabla\bar{\phi}_\Lambda)=f_\Lambda(\phi_\Lambda)\quad\text{in}~\Omega.
\end{equation}
By \eqref{eq:2.03}, $c_{i,\Lambda}(\phi_\Lambda)$ satisfies
\begin{equation}
\label{eq:3.02}
c_{i,\Lambda}(\phi_{\Lambda})=\left(c_{0,\Lambda}(\phi_{\Lambda})\right)^{\lambda_i/\lambda_0}\exp\left(\bar{\mu}_i-z_i\psi-z_i\bar{\phi}_{\Lambda}\right)\quad\text{for}~i=1,\dots,N.
\end{equation}
Since $\psi$ is independent of $\Lambda$ and is continuous on $\overline{\Omega}$, then $\bar{\phi}_\Lambda$ is uniformly bounded if and only if$\phi_\Lambda$ is uniformly bounded. Thus, it suffices to show that $\displaystyle\max_{x\in\overline{\Omega}}\bar{\phi}_\Lambda(x)\leq M_1$ and $\displaystyle\min_{x\in\overline{\Omega}}\bar{\phi}_\Lambda\geq-M_1$ for $\Lambda\geq1$, where $M_1$ is a positive constant independent of $\Lambda$.

Now we prove that $\displaystyle\max_{x\in\overline{\Omega}}\bar{\phi}_\Lambda(x)\leq M_1$ for $\Lambda\geq1$, where $M_1$ is a positive constant independent of $\Lambda$.
Suppose by contradiction that there exists a sequence $\Lambda_k$ with $\displaystyle\lim_{k\to\infty}\Lambda_k=\infty$ such that $\displaystyle\max_{\overline{\Omega}}\bar{\phi}_{\Lambda_k}\geq k$ for $k\in\mathbb N$.
Then there exists $x_k\in\Omega$ such that $\displaystyle\bar{\phi}_{\Lambda_k}(x_k)=\max_{\overline{\Omega}}\bar{\phi}_{\Lambda_k}$ which implies $\nabla\bar{\phi}_{\Lambda_k}\left(x_k\right)=0$ and $\Delta\bar{\phi}_{\Lambda_k}(x_k)\leq0$.
Note that because of the Robin boundary condition of $\bar{\phi}_{\Lambda_k}$, maximum point $x_k$ cannot be located on the boundary $\partial\Omega$ as $k$ sufficiently large.
Hence without loss of generality, we assume each $x_k\in\Omega$ for $k\in\mathbb N$.
For a sake of simplicity, in this proof, we set $c_{i,k}:=c_{i,\Lambda_k}$, $\phi_k:=\phi_{\Lambda_k}$, $f_k:=f_{\Lambda_k}$, and $\bar{\phi}_k:=\bar{\phi}_{\Lambda_k}$.
Hence by equation \eqref{eq:3.01} with $\nabla\bar{\phi}_k(x_k)=0$, $\Delta\bar{\phi}_k(x_k)\leq0$ and function $\epsilon$ is positive, we have
\begin{equation}
\label{eq:3.03}
0\leq-\nabla\epsilon\left(x_k\right)\cdot\nabla\bar{\phi}_k\left(x_k\right)-\epsilon\left(x_k\right)\Delta\bar{\phi}_k\left(x_k\right)=f_k(\phi_k(x_k)).
\end{equation}
Because $\displaystyle f_k(\phi_k)=\sum_{i=1}^Nz_ic_{i,k}(\phi_k)$, we may use \eqref{eq:3.02} and \eqref{eq:3.03} to get
\begin{equation}
\label{eq:3.04}
\begin{aligned}
0&<\sum_{z_i<0}\left(-z_i\right)[c_{0,k}(\phi_k(x_k))]^{\lambda_i/\lambda_0}\exp\left(\bar{\mu}_i-z_i\psi\left(x_k\right)-z_i\bar{\phi}_k\left(x_k\right)\right)\\&\leq\sum_{z_i>0}z_i[c_{0,k}(\phi_k(x_k))]^{\lambda_i/\lambda_0}\exp\left(\bar{\mu}_i-z_i\psi\left(x_k\right)-z_i\bar{\phi}_k\left(x_k\right)\right).
\end{aligned}
\end{equation}
Applying Lemma~\ref{lma:3.1} to \eqref{eq:3.02} for $z_i>0$, we obtain $\displaystyle\lim_{k\to\infty}c_{i,k}(\phi_k(x_k))=0$ for all $i$ with $z_i>0$.
Then \eqref{eq:3.02} and \eqref{eq:3.04} give
\[
\lim_{k\to\infty}\sum_{z_i<0}\left(-z_i\right)[c_{0,k}(\phi_k(x_k))]^{\lambda_i/\lambda_0}\exp\left(\bar{\mu}_i-z_i\psi\left(x_k\right)-z_i\bar{\phi}_k\left(x_k\right)\right)=0,
\]
which means $\displaystyle\lim_{k\to\infty}c_{i,k}(\phi_k(x_k))=0$ for all $i$ with $z_i<0$.
Inserting $x=x_k$ into \eqref{eq:3.02} with $z_i<0$, we can obtain
\[
c_{0,k}(\phi_k(x_k))=[(c_{i,k}(\phi_k(x_k))]^{\lambda_0/\lambda_i}\left[\exp\left(-\bar{\mu}_i+z_i\psi\left(x_k\right)+z_i\bar{\phi}_k\left(x_k\right)\right)\right]^{\lambda_0/\lambda_i}\to0
\]
as $k\to\infty$.
Hence $\displaystyle\lim_{k\to\infty}c_{i,k}(\phi_k(x_k))=0$ for $i=0,1,\dots,N$.
By \eqref{eq:2.03}, \eqref{eq:2.04} and (A2), we get the following contradiction:
\begin{equation}
\label{eq:3.05}
0\geq\lim_{k\to\infty}\frac{\ln c_{0,k}(\phi_k(x_k))}{\Lambda_k}=\lim_{k\to\infty}\left(\lambda_0\tilde{\mu}_0+\frac{\hat{\mu}_0}{\Lambda_k}-\sum_{i=0}^N\lambda_ic_{i,k}(\phi_k(x_k))\right)=\lambda_0\tilde{\mu}_0>0.
\end{equation}
Therefore, we complete the proof to show that $\displaystyle\max_{x\in\overline{\Omega}}\bar{\phi}_\Lambda(x)\leq M_1$ for $\Lambda\geq1$, where $M_1$ is a positive constant independent of $\Lambda$.

It remains to prove that $\displaystyle\min_{x\in\overline{\Omega}}\bar{\phi}_\Lambda\geq-M_1$ for $\Lambda\geq1$, where $M_1$ is a positive independent of $\Lambda$.
Suppose by contradiction that there exists a sequence $\Lambda_k$ with $\displaystyle\lim_{k\to\infty}\Lambda_k=\infty$ such that $\displaystyle\min_{\overline{\Omega}}\bar{\phi}_k\leq-k$ for $k\in\mathbb N$.
Then there exists $x_k\in\Omega$ such that $\displaystyle\bar{\phi}_k(x_k)=\min_{\overline{\Omega}}\bar{\phi}_k$, which implies $\nabla\bar{\phi}_k(x_k)=0$ and $\Delta\bar{\phi}_k(x_k)\geq0$.
Notice that because of the Robin boundary condition of $\bar{\phi}_{\Lambda_k}$, minimum point $x_k$ cannot be located on the boundary $\partial\Omega$ as $k$ sufficiently large. Hence without loss of generality, we assume each $x_k\in\Omega$ for $k\in\mathbb N$.
Thus, as for \eqref{eq:3.04}, we have
\begin{equation}
\label{eq:3.06}
\begin{aligned}
0&<\sum_{z_i>0}z_i[c_{0,k}(\phi_k(x_k))]^{\lambda_i/\lambda_0}\exp\left(\bar{\mu}_i-z_i\psi\left(x_k\right)-z_i\bar{\phi}_k\left(x_k\right)\right)\\&\leq\sum_{z_i<0}(-z_i)[c_{0,k}(\phi_k(x_k))]^{\lambda_i/\lambda_0}\exp\left(\bar{\mu}_i-z_i\psi\left(x_k\right)-z_i\bar{\phi}_k\left(x_k\right)\right).
\end{aligned}
\end{equation}
Applying Lemma~\ref{lma:3.1} to \eqref{eq:3.02} for $z_i<0$, we obtain $\displaystyle\lim_{k\to\infty}c_{i,k}(\phi_k(x_k))=0$ for all $i$ with $z_i<0$.
Then \eqref{eq:3.02} and \eqref{eq:3.06} give $\displaystyle\lim_{k\to\infty}c_{i,k}(\phi_k(x_k))=0$ for all $i$ with $z_i>0$.
Inserting $x=x_k$ into \eqref{eq:3.02} with $z_i>0$, we have $\displaystyle\lim_{k\to\infty}c_{0,k}(\phi_k(x_k))=0$, and hence $\displaystyle\lim_{k\to\infty}c_{i,k}(\phi_k(x_k))=0$ for each $i=0,1,\dots,N$.
As for \eqref{eq:3.05}, we also get a contradiction and complete the proof to show $\displaystyle\min_{x\in\overline{\Omega}}\bar{\phi}_\Lambda(x)\geq-M_1$ for $\Lambda\geq1$, where $M_1$ is a positive constant independent of $\Lambda$.
Therefore, we complete the proof of Lemma~\ref{lma:3.2}.\end{proof}

By Lemma~\ref{lma:3.2}, we have
\begin{lma}
\label{lma:3.3}
There exists positive constant $M_2$ independent of $\Lambda$ such that $\displaystyle\min_{x\in\overline{\Omega}}c_{0,\Lambda}(\phi_\Lambda(x))\geq M_2$ for $\Lambda\geq1$.
\end{lma}
\begin{proof}
Suppose by contrary that there exists $\Lambda_k$ with $\displaystyle\lim_{k\to\infty}\Lambda_k=\infty$, and $x_k\in\overline{\Omega}$ is the minimum point of $c_{0,\Lambda_k}(\phi_{\Lambda_k}(x))$ such that $\displaystyle\lim_{k\to\infty}c_{0,\Lambda_k}(\phi_{\Lambda_k}(x_k))=0$.
Notice that by \eqref{eq:2.04}, $c_{0,\Lambda}(\phi_\Lambda(x))>0$ for $\Lambda\geq1$ and $x\in\overline{\Omega}$.
Then, there exists $N_1\in\mathbb{N}$ such that $\ln c_{0,\Lambda_k}(\phi_{\Lambda_k}(x_k))<0$ for $k>N_1$.
By \eqref{eq:2.03}, \eqref{eq:2.04} at $\phi=\phi_{\Lambda_k}(x_k)$ and assumption (A2), we get
\begin{equation}
\label{eq:3.07}
\sum_{i=0}^N\lambda_ic_{i,\Lambda_k}(\phi_{\Lambda_k}(x_k))>\tilde{\mu}_0+\frac{\hat{\mu}_0}{\Lambda_k\lambda_0}\quad\text{for}~k>N_1.
\end{equation}
On the other hand, by \eqref{eq:3.02} and Lemma~\ref{lma:3.2}, we obtain $\displaystyle\lim_{k\to\infty}c_{i,\Lambda_k}(\phi_{\Lambda_k}(x_k))=0$ for $i=0,1,\dots,N$, which contradicts with \eqref{eq:3.07} as $k\to\infty$.
Hence we complete the proof of Lemma~\ref{lma:3.3}.
\end{proof}

\begin{rmk}
\label{rmk3}
By Lemmas \ref{lma:3.1}--\ref{lma:3.3} and \eqref{eq:2.03}, we have $M_3\leq c_{i,\Lambda}(\phi_\Lambda(x))\leq M_4$ for $x\in\overline{\Omega}$, $\Lambda\geq1$ and $i=0,1,\dots,N$, where $M_3$ and $M_4$ are positive constants independent of $\Lambda$.
Moreover, due to $\displaystyle f_\Lambda(\phi_\Lambda)=\sum_{i=1}^N z_i c_{i,\Lambda}(\phi_\Lambda)$, $\|f_\Lambda(\phi_\Lambda)\|_{L^\infty}\leq M_5$ for $\Lambda\geq1$, where $M_5$ is a positive constant independent of $\Lambda$.
\end{rmk}

\subsection{Convergence of \texorpdfstring{$\phi_\Lambda$}{}}
\label{sec3.2}

Since $\phi_\Lambda$ is the solution of equation $-\nabla\cdot(\epsilon\nabla\phi_\Lambda)=\rho_0+f_\Lambda(\phi_\Lambda)$ in $\Omega$ with the Robin boundary condition $\phi_\Lambda+\eta\frac{\partial\phi_\Lambda}{\partial\nu}=\phi_{bd}$ on $\partial\Omega$, we use the $W^{2,p}$ estimate (cf. \cite[Theorem 15.2]{1959agmon}) to get
\[
\|\phi_\Lambda\|_{W^{2,p}(\Omega)}\leq C(\|\rho_0+f_\Lambda(\phi_\Lambda)\|_{L^p(\Omega)}+\|\phi_{bd}\|_{W^{1,p}(\Omega)})
\]
for all $p>1$, where $C$ is a positive constant independent of $\Lambda$.
Hence by Remark~\ref{rmk3}, we have the uniform bound estimate of $\phi_\Lambda$ in $W^{2,p}$ norm. This implies that there exists a sequence of functions $\{\phi_{\Lambda_k}\}_{k=1}^{\infty}$, with $\displaystyle\lim_{k\to\infty}\Lambda_k=\infty$, such that $\phi_{\Lambda_k}$ converges to $\phi^*$ weakly in $W^{2,p}(\Omega)$.
By Remark~\ref{rmk2} of Proposition~\ref{prop:2.6}, we have the convergence of function $f_{\Lambda_k}$ to $f^*$ in $\mathcal{C}^m[-M_1,M_1]$ for $m\in\mathbb N$ so $\phi^*$ satisfies equation \eqref{eq:1.17} in weak sense, where the positive constant $M_1$ comes from Lemma~\ref{lma:3.2}.
Let $w_k=\phi_{\Lambda_k}-\phi^*$, $c_{i,k}:=c_{i,\Lambda_k}$, and $f_k:=f_{\Lambda_k}$.
Then by Sobolev embedding, $w_k\in\mathcal{C}^{1,\alpha}(\Omega)$ for $\alpha\in(0,1)$, and $\displaystyle\lim_{k\to\infty}\|w_k\|_{\mathcal{C}^{1,\alpha}(\Omega)}=0$.
Moreover, $w_k$ satisfies
\[
-\nabla\cdot(\epsilon\nabla w_k)=f_k(w_k+\phi^*)-f^*(\phi^*)\quad\text{in}~\Omega
\]
with the boundary condition $w_k+\eta\frac{\partial w_k}{\partial\nu}=0$ on $\partial\Omega$.
Using the Schauder's estimate (cf. \cite[Theorem 6.30]{1977gilbarg}) with the mathematical induction, we get
\begin{equation}
\label{eq:3.08}
\begin{aligned}
\|w_k\|_{\mathcal{C}^{m+2,\alpha}(\Omega)}&\leq C\|f_k(w_k+\phi^*)-f^*(\phi^*)\|_{\mathcal{C}^{m,\alpha}(\Omega)}\\&\leq C'\sum_{i=1}^N\left(\|c_{i,k}(w_k+\phi^*)-c_i^*(w_k+\phi^*)\|_{\mathcal{C}^{m,\alpha}(\Omega)}+\|c_i^*(w_k+\phi^*)-c_i^*(\phi^*)\|_{\mathcal{C}^{m,\alpha}(\Omega)}\right),
\end{aligned}
\end{equation}
for all $m\in\mathbb N$ and $\alpha\in(0,1)$, where $C$ and $C'$ are positive constants independent of $k$.
By Proposition~\ref{prop:2.6} and induction hypothesis $\displaystyle\lim_{k\to\infty}\|w_k\|_{\mathcal{C}^{m,\alpha}(\Omega)}=0$, we may use \eqref{eq:3.08} to get $\displaystyle\lim_{k\to\infty}\|w_k\|_{\mathcal{C}^{m+2,\alpha}(\Omega)}=0$, i.e. $\displaystyle\lim_{k\to\infty}\|\phi_{\Lambda_k}-\phi^*\|_{C^{m+2,\alpha}(\Omega)}=0$ for $m\in\mathbb N$ and $\alpha\in (0,1)$.
Therefore, $\phi^*$ is the solution of equation \eqref{eq:1.17} with the Robin boundary condition \eqref{eq:1.08}.

To complete the proof of Theorem~\ref{thm:1.1}, we need to prove
\begin{claim}
\label{claim1}
For any $m\in\mathbb{N}$, we have $\displaystyle\lim_{\Lambda\to\infty}\|\phi_{\Lambda}-\phi^*\|_{\mathcal{C}^m(\Omega)}=0$.
\end{claim}
\begin{proof}
Suppose that there exist the sequences $\{\Lambda_k\}$ and $\{\tilde{\Lambda}_k\}$ tending to infinity such that sequences $\{\phi_{\Lambda_k}\}$ and $\{\phi_{\tilde{\Lambda}_k}\}$ have limits $\phi_1^*$ and $\phi_2^*$, respectively.
It is clear that $\phi_1^*$ and $\phi_2^*$ satisfy the equation \eqref{eq:1.17} with the Robin boundary condition \eqref{eq:1.08}.
Now we want to prove that $\phi_1^*\equiv\phi_2^*$.

Let $u=\phi_1^*-\phi_2^*$. Subtracting \eqref{eq:1.17} with $\phi^*=\phi_2^*$ from that with $\phi^*=\phi_1^*$, we obtain
\[
-\nabla\cdot(\epsilon\nabla u)=f^*(\phi_1^*)-f^*(\phi_2^*)=c(x)u\quad\text{in}~\Omega,
\]
where function $c$ is defined by
\[
c(x)=\begin{cases}\displaystyle\frac{f^*(\phi_1^*(x))-f^*(\phi_2^*(x))}{\phi_1^*(x)-\phi_2^*(x)},&\text{if}~\phi_1^*(x)\neq\phi_2^*(x);\\\displaystyle\frac{\mathrm{d}f^*}{\mathrm{d}\phi}(\phi_1^*(x)),&\text{if}~\phi_1^*(x)=\phi_2^*(x).\end{cases}
\]
By Proposition~\ref{prop:2.7}, we have $c<0$ in $\Omega$ which comes from the fact that if $f$ is a strictly decreasing function on $\mathbb R$, then $\frac{f(\alpha)-f(\beta)}{\alpha-\beta}<0$ for $\alpha\neq\beta$.
Since $\nabla\cdot(\epsilon\nabla u)+c(x)u=0$ with $c<0$ in $\Omega$, it is obvious that $u$ cannot be a nonzero constant.
Then by the strong maximum principle, we have that $u$ attains its nonnegative maximum value and nonnpositive minimum value at the boundary point.
Suppose $u$ has nonnegative maximum value and attains its maximum value at $x^*\in\partial\Omega$.
Then by the boundary condition of $u$ which is $u+\eta\frac{\partial u}{\partial\nu}=0$ on $\partial\Omega$, we get $u(x^*)=-\eta\frac{\partial u}{\partial\nu}(x^*)\leq0$, and hence $u\leq u(x^*)\leq0$ on $\overline{\Omega}$.
Similarly, we obtain $u\geq0$ in $\overline{\Omega}$, and hence $u\equiv0$.
Therefore, we conclude that $\phi_1^*\equiv\phi_2^*$ and complete the proof of Theorem~\ref{thm:1.1}.
\end{proof}

\section{Proof of Theorem~\ref{thm:1.2}}
\label{sec4}

We show the asymptotic limit of $\phi_\Lambda$ for the following two cases: (i) function $\rho_0$ is constant; (ii) function $\rho_0$ is nonconstant and satisfies $\int_{\Omega}\rho_0\,\mathrm{d}x=0$ in Sections~\ref{sec4.1} and \ref{sec4.2}, respectively.
For the case (i), the unique constant solution $\phi_\Lambda$ follows from Proposition~\ref{prop:2.2} and \ref{prop:2.4}, and the uniform convergence of $\phi_\Lambda$ can be obtained from Remark~\ref{rmk2} of Proposition~\ref{prop:2.6}.
For the case (ii), the existence, uniqueness, and regularity of $\phi_\Lambda$ can be proved by the standard Direct method, which is stated in Appendix I.
Unlike Lemma~\ref{lma:3.2}, we use Hopf's lemma to prove the uniform boundedness of $\phi_\Lambda$ with the Neumann boundary condition (See Lemma~\ref{lma:4.1}).
By Lemmas~\ref{lma:3.1},~\ref{lma:3.3} and~\ref{lma:4.1}, we can follow the same argument in Section~\ref{sec3.2} to complete the proof of Theorem~\ref{thm:1.2}(ii).

\subsection{Constant \texorpdfstring{$\rho_0$}{rho0}}
\label{sec4.1}

The following is the proof of Theorem~\ref{thm:1.2}(i).

\begin{proof}[Proof of Theorem~\ref{thm:1.2}(i)]Firstly, we fix $\Lambda>0$ and suppose that $\rho_0$ is constant.
Then by the strict monotonicity and unboundedness of $f_{\Lambda}$ (see Propositions~\ref{prop:2.2} and~\ref{prop:2.4}), there exists a unique $\hat{\phi}_{\Lambda}\in\mathbb{R}$ such that $\rho_0+f_{\Lambda}(\hat{\phi}_{\Lambda})=0$.
Hence $\phi_{\Lambda}\equiv\hat{\phi}_{\Lambda}$ is a constant solution of equation \eqref{eq:1.12} with the Neumann boundary condition \eqref{eq:1.09}.
For the uniqueness of solution $\phi_\Lambda$ up to a constant, suppose that $\phi_{1,\Lambda}$ and $\phi_{2,\Lambda}$ are the solutions of equation \eqref{eq:1.12} with the Neumann boundary condition \eqref{eq:1.09}.
We subtract the equation \eqref{eq:1.07} for $\phi=\phi_{2,\Lambda}$ from that for $\phi=\phi_{1,\Lambda}$ to get
\begin{equation}
\label{eq:4.01}
-\nabla\cdot(\epsilon\nabla(\phi_{1,\Lambda}-\phi_{2,\Lambda}))=\sum_{i=1}^Nz_i\left[c_{i,\Lambda}(\phi_{1,\Lambda})-c_{i,\Lambda}(\phi_{2,\Lambda})\right]=f_\Lambda(\phi_{1,\Lambda})-f_\Lambda(\phi_{2,\Lambda})\quad\text{in}~\Omega.
\end{equation} Multiplying the equation \eqref{eq:4.01} by $\phi_{1,\Lambda}-\phi_{2,\Lambda}$ and integrating it over $\Omega$, we have
\[-\int_\Omega(\phi_{1,\Lambda}-\phi_{2,\Lambda})\nabla\cdot(\epsilon\nabla(\phi_{1,\Lambda}-\phi_{2,\Lambda}))\mathrm{d}x=\int_\Omega(\phi_{1,\Lambda}-\phi_{2,\Lambda})[f_\Lambda(\phi_{1,\Lambda})-f_\Lambda(\phi_{2,\Lambda})]\mathrm{d}x.\]
Then using integration by parts and the Neumann boundary condition, we obtain
\[\int_\Omega\epsilon|\nabla(\phi_{1,\Lambda}-\phi_{2,\Lambda})|^2\mathrm{d}x=\int_\Omega(\phi_{1,\Lambda}-\phi_{2,\Lambda})[f_\Lambda(\phi_{1,\Lambda})-f_\Lambda(\phi_{2,\Lambda})]\mathrm{d}x\leq0.\]
Here we have used the monotone decreasing of $f_\Lambda$ (see Proposition~\ref{prop:2.2}) and function $\epsilon$ is positive.
Thus, $\phi_{1,\Lambda}-\phi_{2,\Lambda}\equiv C$ for some constant $C$, and we get the uniqueness of $\phi_\Lambda$ up to a constant.
This implies that $\phi_\Lambda$ is a constant solution and satisfies $\phi_\Lambda=\hat{\phi}_\Lambda+C$.
Due to the uniqueness of equation $\rho_0+f_\Lambda(\hat{\phi}_\Lambda)=0$, we obtain the constant $C=0$ and $\phi_{\Lambda}\equiv\hat{\phi}_{\Lambda}$. By Propositions~\ref{prop:2.7} and \ref{prop:2.8}, $f^*$ is strictly decreasing and the range of $f^*$ is exactly the interval $(m^*,M^*)$, there exists a unique $\phi^*\in\mathbb R$ such that $\rho_0+f^*(\phi^*)=0$ for $\rho_0\in(-M^*,-m^*)$.
By Remark~\ref{rmk2} of Proposition~\ref{prop:2.6}, we conclude that $\phi_\Lambda$ is uniformly bounded and $\displaystyle\lim_{\Lambda\to\infty}\hat{\phi}_{\Lambda}=\phi^*$.
The proof of Theorem~\ref{thm:1.2}(i) is complete.
\end{proof}

\subsection{Uniform boundedness of \texorpdfstring{$\phi_{\Lambda}$}{} and \texorpdfstring{$c_{0,\Lambda}$}{}}
\label{sec4.2}

Under the Neumann boundary condition, Lemma~\ref{lma:3.1} also holds true so $\displaystyle\max_{x\in\overline{\Omega}}c_{0,\Lambda}(\phi_\Lambda(x))\leq M_0$ for $\Lambda\geq1$, where $M_0$ is a positive constant independent of $\Lambda$.
However, due to the Neumann boundary condition, we have to use the Hopf's Lemma to modify the proof of  Lemma~\ref{lma:3.2} for the uniform boundedness of $\phi_\Lambda$ (see Lemma~\ref{lma:4.1}).Then  as for Lemma~\ref{lma:3.3}, we can use Lemma~\ref{lma:4.1} to prove $\displaystyle\min_{x\in\overline{\Omega}}c_{0,\Lambda}(\phi_\Lambda(x))\geq M_2$ for $\Lambda\geq1$, where $M_2$ is a positive constant independent of $\Lambda$. Now we state and prove Lemma~\ref{lma:4.1} as follows.

\begin{lma}
\label{lma:4.1}
Under the same hypotheses of Theorem~\ref{thm:1.2}(ii), let $\phi_{\Lambda}$ be the solution of equation \eqref{eq:1.12} with the Neumann boundary condition \eqref{eq:1.09}.
Then there exists a positive constant $M_6$ independent of $\Lambda$ such that $\|\phi_{\Lambda}\|_{L^\infty(\Omega)}\leq M_6$ for $\Lambda\geq1$.
\end{lma}
\begin{proof}
Based on $\int_{\Omega}\rho_0\,\mathrm{d}x=0$, let $\psi_1$ and $\psi_2$ be the solution of Poisson equation $-\nabla\cdot(\epsilon\nabla\psi)=\rho_0$ in $\Omega$ with the Neumann boundary condition \eqref{eq:1.09} such that $\displaystyle\min_{\overline{\Omega}}\psi_1=0$ and $\displaystyle\max_{\overline{\Omega}}\psi_2=0$.
Note that $\psi_1$ and $\psi_2$ are independent of $\Lambda$, and up to a constant.
Let $\bar{\phi}_{i,\Lambda}=\phi_\Lambda-\psi_i$ for $i=1,2$. Then $\bar{\phi}_{i,\Lambda}$ satisfies
\begin{equation}
\label{eq:4.02}
-\nabla\cdot(\epsilon\nabla\bar{\phi}_{i,\Lambda})=f_\Lambda(\phi_\Lambda)\quad\text{in}~\Omega,\quad\text{for}~i=1,2.
\end{equation}
Since $\psi_i$ is continuous on $\overline{\Omega}$, then $\bar{\phi}_{i,\Lambda}$ is uniformly bounded if and only if $\phi_\Lambda$ is uniformly bounded.
Thus, it suffices to show that $\displaystyle\max_{x\in\overline{\Omega}}\bar{\phi}_{1,\Lambda}(x)\leq M_7$ and $\displaystyle\min_{x\in\overline{\Omega}}\bar{\phi}_{2,\Lambda}(x)\geq-M_7$ for $\Lambda\geq1$, where $M_7$ is a positive constant independent of $\Lambda$.

Firstly, we prove that there exists a positive constant $M_7$ independent of $\Lambda$ such that $\displaystyle\max_{x\in\overline{\Omega}}\bar{\phi}_{1,\Lambda}(x)\leq M_7$ for $\Lambda\geq1$.
Suppose by contradiction that there exists a sequence $\Lambda_k$ with $\displaystyle\lim_{k\to\infty}\Lambda_k=\infty$ such that $\displaystyle\max_{\overline{\Omega}}\bar{\phi}_{1,\Lambda_k}\geq k$ for $k\in\mathbb N$.
As for the proof of Lemma~\ref{lma:3.2}, we need to prove
\begin{claim}
\label{claim2}
$\bar{\phi}_{1,\Lambda_k}$ attains its maximum value at $x_k\in\Omega$ for sufficient large $k$.
\end{claim}
\begin{proof}
Suppose by contradiction that the maximum point $x_k$ of $\bar{\phi}_{1,\Lambda_k}$ is located on the boundary $\partial\Omega$, i.e., $x_k\in\partial\Omega$ for $k\in\mathbb N$.
By Theorem~\ref{thm:1.2}(i), there exists a unique constant $\hat{\phi}_\Lambda\in\mathbb R$ such that $f_\Lambda(\hat{\phi}_\Lambda)=0$ and $|\hat{\phi}_\Lambda|\leq M$ for $\Lambda\geq1$, where $M$ is a positive constant independent of $\Lambda$.
For notation convenience, let $\tilde{\phi}_k:=\bar{\phi}_{1,\Lambda_k}-\hat{\phi}_{\Lambda_k}=\phi_{\Lambda_k}-\psi_1-\hat{\phi}_{\Lambda_k}$, $\phi_k:=\phi_{\Lambda_k}$, $\bar{\phi}_k:=\bar{\phi}_{1,\Lambda_k}$, $\hat{\phi}_k:=\hat{\phi}_{\Lambda_k}$, and $f_k:=f_{\Lambda_k}$.
Since $\hat{\phi}_k$ is a constant with $f_k(\hat{\phi}_k)=0$ and $|\hat{\phi}_k|\leq M$ for $k\in\mathbb N$, then equation \eqref{eq:4.02} can be transformed to
\begin{equation}
\label{eq:4.03}
-\nabla\cdot(\epsilon\nabla\tilde{\phi}_k)=f_k(\tilde{\phi}_k+\hat{\phi}_k+\psi_1)-f_k(\hat{\phi}_k)=c(x)(\tilde{\phi}_k+\psi_1)\quad\text{in}~\Omega,
\end{equation}
where function $c=c(x)$ is defined by
\[
c(x)=\begin{cases}
\displaystyle\frac{f_k(\tilde{\phi}_k(x)+\hat{\phi}_k+\psi_1(x))-f_k(\hat{\phi}_k)}{(\tilde{\phi}_k(x)+\hat{\phi}_k+\psi_1(x))-\hat{\phi}_k},&\text{if~}\tilde{\phi}_k(x)+\psi_1(x)\neq0;\\\displaystyle\frac{\mathrm{d}f_k}{\mathrm{d}\phi}(\hat{\phi}_k),&\text{if}~\tilde{\phi}_k(x)+\psi_1(x)=0.
\end{cases}
\]
By equation \eqref{eq:4.03}, we have $\nabla\cdot(\epsilon\nabla\tilde{\phi}_k)+c(x)\tilde{\phi}_k=-c(x)\psi_1\geq0$ in $\Omega$. Here we have used that $\displaystyle\min_{\overline{\Omega}}\psi_1=0$ and $c<0$ in $\Omega$ because of Proposition~\ref{prop:2.2}.
Notice that if $f$ is a strictly decreasing function on $\mathbb R$, then $\frac{f(\alpha)-f(\beta)}{\alpha-\beta}<0$ for $\alpha\neq\beta$.
Since $\hat{\phi}_k$ is a constant, $\tilde{\phi}_k$ and $\bar{\phi}_k$ have the same maximum point $x_k$ for $k\in\mathbb N$.
Moreover, $\tilde{\phi}_k(x_k)=\bar{\phi}_{1,\Lambda_k}(x_k)-\hat{\phi}_{\Lambda_k}\geq k-M>0$ if $k>M$.
Hence by Hopf's lemma, we get $\frac{\partial\tilde{\phi}_k}{\partial\nu}(x_k)>0$, which contradicts with the Neumann boundary condition.
Thus, $\bar{\phi}_k$ attains its maximum value at interior point $x_k\in\Omega$ for $k>M$, and we complete the proof of Claim~\ref{claim2}.
\end{proof}
\noindent Claim~\ref{claim2} implies $\nabla\bar{\phi}_{1,\Lambda_k}(x_k)=0$ and $\Delta\bar{\phi}_{1,\Lambda_k}(x_k)\leq0$ for $k>M$.
Then as for the proof of \eqref{eq:3.03}--\eqref{eq:3.05}, we conclude that $\displaystyle\max_{x\in\overline{\Omega}}\bar{\phi}_{1,\Lambda}(x)\leq M_7$ for $\Lambda\geq1$ which gives $\displaystyle\max_{x\in\overline{\Omega}}\phi_\Lambda(x)\leq M_6$ for $\Lambda\geq1$, where $M_6$ and $M_7$ are positive constants independent of $\Lambda$.

It remains to prove that $\displaystyle\min_{x\in\overline{\Omega}}\bar{\phi}_{2,\Lambda}(x)\geq-M_7$ for $\Lambda\geq1$, where $M_7$ is a positive constant independent of $\Lambda$.
Suppose by contradiction that there exists a sequence $\Lambda_k$ with $\displaystyle\lim_{k\to\infty}\Lambda_k=\infty$ such that $\displaystyle\min_{\overline{\Omega}}\bar{\phi}_{2,\Lambda_k}\leq-k$ for $k\in\mathbb N$.
As for the proof of Lemma~\ref{lma:3.2}, we need to prove
\begin{claim}
\label{claim3}
$\bar{\phi}_{2,\Lambda_k}$ attains its minimum value at $x_k\in\Omega$ for sufficient large $k$.
\end{claim}
\begin{proof}
Suppose by contradiction that the minimum point $x_k$ of $\bar{\phi}_{2,\Lambda_k}$ is located on the boundary $\partial\Omega$, i.e., $x_k\in\partial\Omega$ for $k\in\mathbb N$.
For a sake simplicity, we use the similar notations of Claim~\ref{claim2} and let $\tilde{\phi}_k:=\bar{\phi}_{2,\Lambda_k}-\hat{\phi}_{\Lambda_k}=\phi_{\Lambda_k}-\psi_2-\hat{\phi}_{\Lambda_k}$, $\phi_k:=\phi_{\Lambda_k}$, $\bar{\phi}_k:=\bar{\phi}_{2,\Lambda_k}$, $\hat{\phi}_k:=\hat{\phi}_{\Lambda_k}$.
Then equation \eqref{eq:4.02} becomes
\begin{equation}
\label{eq:4.04}
-\nabla\cdot(\epsilon\nabla\tilde{\phi}_k)=f_k(\tilde{\phi}_k+\hat{\phi}_k+\psi_2)-f_k(\hat{\phi}_k)=c(x)(\tilde{\phi}_k+\psi_2)\quad\text{in}~\Omega,
\end{equation}
where function $c=c(x)$ is defined by
\[c(x)=\begin{cases}
\displaystyle \frac{f_k(\tilde{\phi}_k(x)+\hat{\phi}_k+\psi_2(x))-f_k(\hat{\phi}_k)}{(\tilde{\phi}_k(x)+\hat{\phi}_k+\psi_2(x))-\hat{\phi}_k},&\text{if~}\tilde{\phi}_k(x)+\psi_2(x)\neq0;\\\displaystyle\frac{\mathrm{d}f_k}{\mathrm{d}\phi}(\hat{\phi}_k),&\text{if}~\tilde{\phi}_k(x)+\psi_2(x)=0.
\end{cases}
\]
From equation \eqref{eq:4.04}, we have $\nabla\cdot(\epsilon\nabla\tilde{\phi}_k)+c(x)\tilde{\phi}_k=-c(x)\psi_2\leq0$ in $\Omega$.
Here we have used that $\displaystyle\max_{\overline{\Omega}}\psi_2=0$ and $c<0$ in $\Omega$ because of Proposition~\ref{prop:2.2}.
Notice that if $f$ is a strictly decreasing function on $\mathbb R$, then $\frac{f(\alpha)-f(\beta)}{\alpha-\beta}<0$ for $\alpha\neq\beta$.
Since $\hat{\phi}_k$ is a constant, $\tilde{\phi}_k$ and $\bar{\phi}_k$ have the same minimum point $x_k$ for $k\in\mathbb N$.
Moreover, $\tilde{\phi}_k(x_k)=\bar{\phi}_{2,\Lambda_k}(x_k)-\hat{\phi}_{\Lambda_k}\leq -k+M<0$ if $k>M$.
Hence by Hopf's lemma, we get $\frac{\partial\tilde{\phi}_k}{\partial\nu}(x_k)<0$, which contradicts with the Neumann boundary condition.
Thus, $\bar{\phi}_k$ attains its minimum value at interior point $x_k\in\Omega$ for $k>M$, and we complete the proof of Claim~\ref{claim3}.
\end{proof}
\noindent Claim~\ref{claim3} implies $\nabla\bar{\phi}_{2,\Lambda_k}(x_k)=0$ and $\Delta\bar{\phi}_{2,\Lambda_k}(x_k)\geq0$ for $k>M$.
Then as for \eqref{eq:3.06},
we conclude that $\displaystyle\min_{x\in\overline{\Omega}}\bar{\phi}_{2,\Lambda}(x)\geq-M_7$ for $\Lambda\geq1$ which gives $\displaystyle\lim_{x\in\overline{\Omega}}\phi_\Lambda(x)\geq-M_6$ for $\Lambda\geq1$, where $M_6$ and $M_7$ are positive constants independent of $\Lambda$.
Therefore the proof of Lemma~\ref{lma:4.1} is complete.\end{proof}

The following is the proof of Theorem~\ref{thm:1.2}(ii).
Because $\phi_\Lambda$ is the solution of $-\nabla\cdot(\epsilon\nabla\phi_\Lambda)=\rho_0+f_\Lambda(\phi_\Lambda)$ in $\Omega$ with the Neumann boundary condition $\frac{\partial\phi_\Lambda}{\partial\nu}=0$ on $\partial\Omega$, we use the $W^{2,p}$ estimate (cf. \cite[Theorem~15.2]{1959agmon}) to get
\[\|\phi_{\Lambda}\|_{W^{2,p}(\Omega)}\leq C\|\rho_0+f_\Lambda(\phi_\Lambda)\|_{L^p(\Omega)}\]
for all $p>1$, where $C$ is a positive constant independent of $\Lambda$.
Then as for Remark~\ref{rmk3} of Section~\ref{sec3.1}, we use Lemma~\ref{lma:4.1} to get the uniform bound estimate of $\phi_\Lambda$ in $W^{2,p}$ norm, and hence there exists a sequence of function $\{\phi_{\Lambda_k}\}_{k=1}^\infty$ such that $\phi_{\Lambda_k}$ converges to $\phi^*$ weakly in $W^{2,p}(\Omega)$.
By Remark~\ref{rmk2} of Proposition~\ref{prop:2.6}, $\phi^*$ satisfies the equation \eqref{eq:1.17} in weak sense.
Let $w_k=\phi_{\Lambda_k}-\phi^*$, $c_{i,k}:=c_{i,\Lambda_k}$, and $f_k:=f_{\Lambda_k}$.
Then by Sobolev embedding, $w_k\in\mathcal{C}^{1,\alpha}(\Omega)$ for $\alpha\in(0,1)$, and $\displaystyle\lim_{k\to\infty}\|w_k\|_{\mathcal{C}^{1,\alpha}(\Omega)}=0$.
Moreover, $w_k$ satisfies
\[
-\nabla\cdot(\epsilon\nabla w_k)=f_k(w_k+\phi^*)-f^*(\phi^*)
\quad\text{in}~\Omega\]
with the boundary condition $\frac{\partial w_k}{\partial\nu}=0$ on $\partial\Omega$.
Using the Schauder's estimate (cf. \cite[Theorem 6.30]{1977gilbarg}) with the mathematical induction, we get
\begin{equation}
\label{eq:4.05}
\begin{aligned}
\|w_k\|_{\mathcal{C}^{m+2,\alpha}(\Omega)}&\leq C\|f_k(w_k+\phi^*)-f^*(\phi^*)\|_{\mathcal{C}^{m,\alpha}(\Omega)}\\&\leq C'\sum_{i=1}^N\left(\|c_{i,k}(w_k+\phi^*)-c_i^*(w_k+\phi^*)\|_{\mathcal{C}^{m,\alpha}(\Omega)}+\|c_i^*(w_k+\phi^*)-c_i^*(\phi^*)\|_{\mathcal{C}^{m,\alpha}(\Omega)}\right),
\end{aligned}
\end{equation}
for $m\in\mathbb N$ and $\alpha\in(0,1)$, where $C$ and $C'$ are positive constants independent of $k$.
By Proposition~\ref{prop:2.6} and induction hypothesis $\displaystyle\lim_{k\to\infty}\|w_k\|_{\mathcal{C}^{m,\alpha}(\Omega)}=0$, we may use \eqref{eq:4.04} to get  $\displaystyle\lim_{k\to\infty}\|w_k\|_{\mathcal{C}^{m+2,\alpha}(\Omega)}=0$, i.e. $\displaystyle\lim_{k\to\infty}\|\phi_{\Lambda_k}-\phi^*\|_{C^{m+2,\alpha}(\Omega)}=0$ for $m\in\mathbb N$ and $\alpha\in (0,1)$.
Therefore, $\phi^*$ is the solution of \eqref{eq:1.17} with the Neumann boundary condition \eqref{eq:1.09}.

To complete the proof of Theorem~\ref{thm:1.2}(ii), we need to prove
\begin{claim}
\label{claim4}
For any $m\in\mathbb N$, we have $\displaystyle\lim_{\Lambda\to\infty}\|\phi_\Lambda-\phi^*\|_{\mathcal{C}^m(\Omega)}=0$.
\end{claim}
\begin{proof}
Suppose that there exist sequences $\{\Lambda_k\}$ and $\{\tilde{\Lambda}_k\}$ tending to infinity such that sequences $\{\phi_{\Lambda_k}\}$ and $\{\phi_{\tilde{\Lambda}_k}\}$ have limits $\phi_1^*$ and $\phi_2^*$, respectively.
It is clear that $\phi_1^*$ and $\phi_2^*$ satisfy the equation \eqref{eq:1.17} with the Neumann boundary condition \eqref{eq:1.09}.
Now we want to prove that $\phi_1^*\equiv\phi_2^*$.
Let $u=\phi_1^*-\phi_2^*$. Subtracting the equation \eqref{eq:1.17} with $\phi^*=\phi_2^*$ from that with $\phi^*=\phi_1^*$, we obtain
\[-\nabla\cdot(\epsilon\nabla u)=f^*(\phi_1^*)-f^*(\phi_2^*)=c(x)u\quad\text{in}~\Omega,\]where function $c$ is defined by
\[
c(x)=\begin{cases}\displaystyle\frac{f^*(\phi_1^*(x))-f^*(\phi_2^*(x))}{\phi_1^*(x)-\phi_2^*(x)},&\text{if}~\phi_1^*(x)\neq\phi_2^*(x);\\\displaystyle\frac{\mathrm{d}f^*}{\mathrm{d}\phi}(\phi_1^*(x)),&\text{if}~\phi_1^*(x)=\phi_2^*(x).\end{cases}
\]
By Proposition~\ref{prop:2.7}, we have $c<0$ in $\Omega$.
Here we have used the fact that $\frac{f(\alpha)-f(\beta)}{\alpha-\beta}<0$ for $\alpha\neq\beta$ if $f$ is strictly decreasing on $\mathbb R$.
Since $\nabla\cdot(\epsilon\nabla u)+c(x)u=0$ with $c<0$ in $\Omega$, then it is obvious that $u$ cannot be a nonzero constant.
Then by the strong maximum principle, we have that $u$ attains its nonnegative maximum value and nonnpositive minimum value at the boundary point.
Suppose $u$ has positive maximum value which is attained at $x^*\in\partial\Omega$.
Then by Hopf's lemma, we get $\frac{\partial u}{\partial\nu}(x^*)>0$, which contradicts with the Neumann boundary condition.
Thus $u\leq0$ in $\Omega$.
Similarly, $u\geq 0$ in $\Omega$. Therefore, $u\equiv0$, i.e. $\phi^*_1\equiv\phi^*_2$ and we complete the proof of Claim~\ref{claim4} and Theorem~\ref{thm:1.2}(ii).\end{proof}

\section{Numerical results}
\label{sec5}

Here we present numerical results to support Theorems~\ref{thm:1.1} and \ref{thm:1.2}.
Throughout this section, we assume the one-dimensional domain $\Omega=(-1,1)$, $N=3$, $z_0=0$, $z_1=1$, $z_2=-1$, $z_3=2$, $\eta=\epsilon=0.1$, and $\bar{\mu}_i=\tilde{\mu}_0=1$ for all $i=0,1,\dots,N$.
By the method of \cite{2015elbaghdady}, we employ the Legendre--Gauss--Lobatto (LGL) points $\{x_k\}_{k=0}^L$ as the partition of the interval $[-1,1]$ to discrtize equations \eqref{eq:1.11} and \eqref{eq:1.12} as the following algebraic equations.
\begin{align}
\label{eq:5.01}&\ln c_{0,\Lambda}(x_k)+\Lambda\lambda_0\sum_{j=0}^N\lambda_j(c_{0,\Lambda}(x_k))^{\lambda_j/\lambda_0}\exp(\bar{\mu}_j-z_j\phi_\Lambda(x_k))=\mu_0,\quad\text{for}~k=0,\dots,L,\\
&\label{eq:5.02}
-\epsilon\text{D}_{L+1}^2\boldsymbol\phi_\Lambda=\boldsymbol\rho_0+\sum_{i=1}^Nz_i{\bf c}_{0,\Lambda}^{\lambda_i/\lambda_0}\circ\exp(\bar{\mu}_i-z_i\boldsymbol\phi_\Lambda),
\end{align}
where $\boldsymbol\rho_0=[\rho_0(x_0),\dots,\rho_0(x_L)]^{\mathsf{T}}$, ${\bf c}_{0,\Lambda}=[c_{0,\Lambda}(x_0),\dots,c_{0,\Lambda}(x_L)]^{\mathsf{T}}$, $\boldsymbol\phi_\Lambda=[\phi_\Lambda(x_0),\dots,\phi_\Lambda(x_L)]^{\mathsf{T}}$ as column vectors, $\circ$ is the Hadamard product, and $D_{L+1}=[d_{ij}]_{0\leq i,j\leq L}$ is the differentiation matrix which satisfies $D_{L+1}\boldsymbol\phi_\Lambda=[\phi_\Lambda'(x_0),\dots,\phi_\Lambda'(x_L)]^{\mathsf{T}}$.
Besides, the Robin boundary condition \eqref{eq:1.08} and the Neumann boundary condition \eqref{eq:1.09} are also discretized as
\begin{align}
\label{eq:5.03}
&(\text{Robin})\qquad~\phi(x_0)-\eta\sum_{k=0}^Ld_{0k}\phi(x_k)=\phi_{bd}\left(-1\right),\quad\phi(x_L)+\eta\sum_{k=0}^Ld_{Lk}\phi(x_k)=\phi_{bd}\left(1\right),
\\
\label{eq:5.04}
&(\text{Neumann})\quad\sum_{k=0}^Ld_{0k}\phi(x_k)=0,\qquad\sum_{k=0}^Ld_{Lk}\phi(x_k)=0.
\end{align}
For the Robin and Neumann boundary conditions, we replace the first and last equations of \eqref{eq:5.02} by equations \eqref{eq:5.03} and\eqref{eq:5.04}, respectively.
As for \eqref{eq:5.01} and \eqref{eq:5.02}, we discrtize the equations \eqref{eq:1.15} and \eqref{eq:1.17} as follows.
\begin{align}
\label{eq:5.05}
&\sum_{i=0}^N\lambda_i(c_0^*(x_k))^{\lambda_i/\lambda_0}\exp\left(\bar{\mu}_i-z_i\phi^*(x_k)\right)=\tilde{\mu}_0,\quad\text{for}~k=0,\dots,L,\\&
\label{eq:5.06}
-\epsilon\,\text{D}_{L+1}^2\boldsymbol{\phi}^*=\boldsymbol\rho_0+\sum_{i=1}^Nz_i({\bf c}_0^*)^{\lambda_i/\lambda_0}\circ\exp\left(\bar{\mu}_i-z_i\boldsymbol{\phi}^*\right),
\end{align}
where ${\bf c}_0^*=[c_0^*(x_0),\dots,c_0^*(x_L)]^{\mathsf{T}}$ and $\boldsymbol\phi^*=[\phi^*(x_0),\dots,\phi^*(x_L)]^{\mathsf{T}}$.
Similarly, for the Robin and Neumann boundary conditions, we replace the first and last equations of \eqref{eq:5.06} by equations \eqref{eq:5.03} and \eqref{eq:5.04}, respectively.

\begin{table}[!t]
\centering
\begin{tabular}{c|c|c}
$\Lambda$&(a) $\lambda_i=1$ for all $i$&(b) $\lambda_1=2$, $\lambda_2=1.5$, $\lambda_3=1$\\\hline
$1$&$5.9950$e-$01$&$6.4144$e-$01$\\
$10$&$1.4501$e-$01$&$1.3755$e-$01$\\
$10^2$&$1.7715$e-$02$&$1.5933$e-$02$\\
$10^3$&$1.8144$e-$03$&$1.6200$e-$03$\\
$10^4$&$1.8189$e-$04$&$1.6227$e-$04$\\
$10^5$&$1.8193$e-$05$&$1.6230$e-$05$\\
\end{tabular}
\caption{The maximum norms $\left\|\phi_\Lambda-\phi^*\right\|_\infty$ with the Robin boundary condition \eqref{eq:5.03}, where $\eta=0.1$, $\phi_{bd}\left(-1\right)=-2$ and $\phi_{bd}\left(1\right)=3$.}
\label{table1}

~

\begin{tabular}{c|c|c}
$\Lambda$&(a) $\lambda_i=1$ for all $i$&(b) $\lambda_1=2$, $\lambda_2=1.5$, $\lambda_3=1$\\\hline
$1$&$6.2243$e-$01$&$6.4429$e-$01$\\
$10$&$1.4654$e-$01$&$1.3285$e-$01$\\
$10^2$&$1.7727$e-$02$&$1.5170$e-$02$\\
$10^3$&$1.8133$e-$03$&$1.5393$e-$03$\\
$10^4$&$1.8175$e-$04$&$1.5416$e-$04$\\
$10^5$&$1.8179$e-$05$&$1.5419$e-$05$\\
\end{tabular}
\caption{The maximum norms $\left\|\phi_\Lambda-\phi^*\right\|_\infty$ with the Dirichlet boundary condition \eqref{eq:5.03}, where $\eta=0$, $\phi_{bd}\left(-1\right)=-2$ and $\phi_{bd}\left(1\right)=3$.}
\label{table2}

~

\begin{tabular}{c|c|c}
$\Lambda$&(a) $\lambda_i=1$ for all $i$&(b) $\lambda_1=2$, $\lambda_2=1.5$, $\lambda_3=1$\\\hline
$1$&$3.0027$e-$01$&$5.3706$e-$01$\\
$10$&$9.4618$e-$02$&$1.6046$e-$01$\\
$10^2$&$1.2565$e-$02$&$2.0737$e-$02$\\
$10^3$&$1.3010$e-$03$&$2.1385$e-$03$\\
$10^4$&$1.3057$e-$04$&$2.1452$e-$04$\\
$10^5$&$1.3062$e-$05$&$2.1459$e-$05$\\
\end{tabular}
\caption{The maximum norms $\left\|\phi_\Lambda-\phi^*\right\|_\infty$ with the Neumann boundary condition \eqref{eq:5.04}, and $\rho_0(x)=x^3$.}
\label{table3}
\end{table}
Using the command {\ttfamily fsolve} in Matlab, we numerically solve the equations \eqref{eq:5.01}--\eqref{eq:5.06} and have the following results.
For various $\Lambda$ with $\rho_0=0$ and the Robin boundary condition with $\eta=0.1$, the profiles of $\phi_\Lambda$ and $\phi^*$ are represented in Figure~\ref{fig2}, and the maximum norms $\|\phi_\Lambda-\phi^*\|_\infty$ are listed in Table~\ref{table1}.
For various $\Lambda$ with $\rho_0=0$ and the Dirichlet boundary condition (the Robin boundary condition with $\eta=0$), the profiles of $\phi_\Lambda$ and $\phi^*$ are sketched in Figure~\ref{fig3}, and the maximum norms $\|\phi_\Lambda-\phi^*\|_\infty$ are shown in Table~\ref{table2}.
These results are consistent with Theorem~\ref{thm:1.1}.
In addition, for the Neumann boundary condition, we choose $\rho_0(x)=x^3$ to fulfill the constraint $\int_{-1}^1\rho_0\mathrm{d}x=0$. For various $\Lambda$, the profiles of $\phi_\Lambda$ and $\phi^*$ are presented in Figure~\ref{fig4} and the maximum norm $\left\|\phi_\Lambda-\phi^*\right\|_\infty$ are expressed in Table~\ref{table3} which support Theorem~\ref{thm:1.2}.
\begin{figure}[!h]
\includegraphics[scale=0.68]{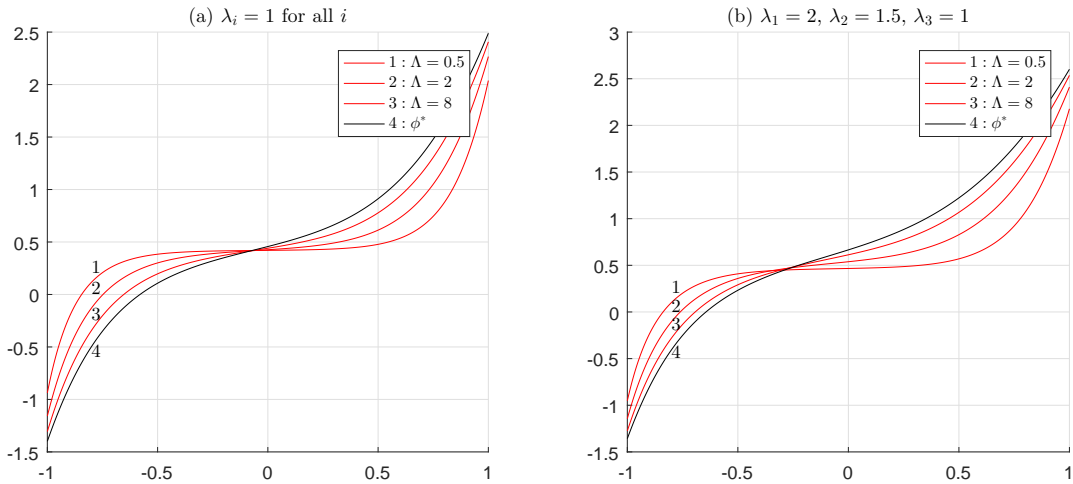}
\caption{The numerical profiles of $\phi_\Lambda$ and $\phi^*$ with the Robin boundary condition, where the curves 1--3 are the profiles of $\phi_\Lambda$ with $\Lambda=0.5,2,8$ and the curve 4 is the profile of $\phi^*$.
In Figure~\ref{fig2} (a), $\lambda_i=1$ for all $i$. In Figure~\ref{fig2} (b),  $\lambda_1=2$, $\lambda_2=1.5$, $\lambda_3=1$.}
\label{fig2}
\end{figure}

\begin{figure}
\centering
\includegraphics[scale=0.6]{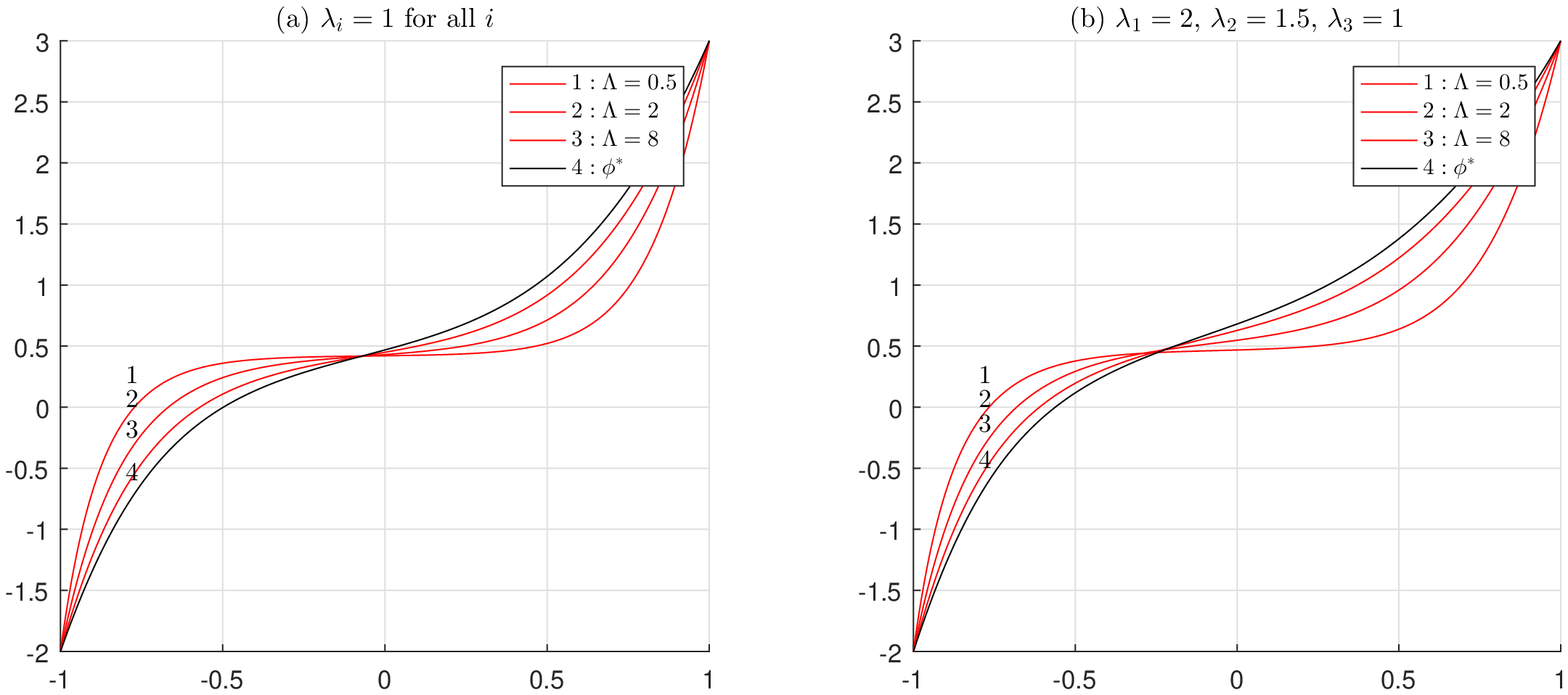}
\caption{The numerical profiles of $\phi_\Lambda$ and $\phi^*$ with the Dirichlet boundary condition, where curves 1--3 are the profiles of $\phi_\Lambda$ with $\Lambda=0.5,2,8$, and curve 4 is the profile of $\phi^*$.}
\label{fig3}
\includegraphics[scale=0.6]{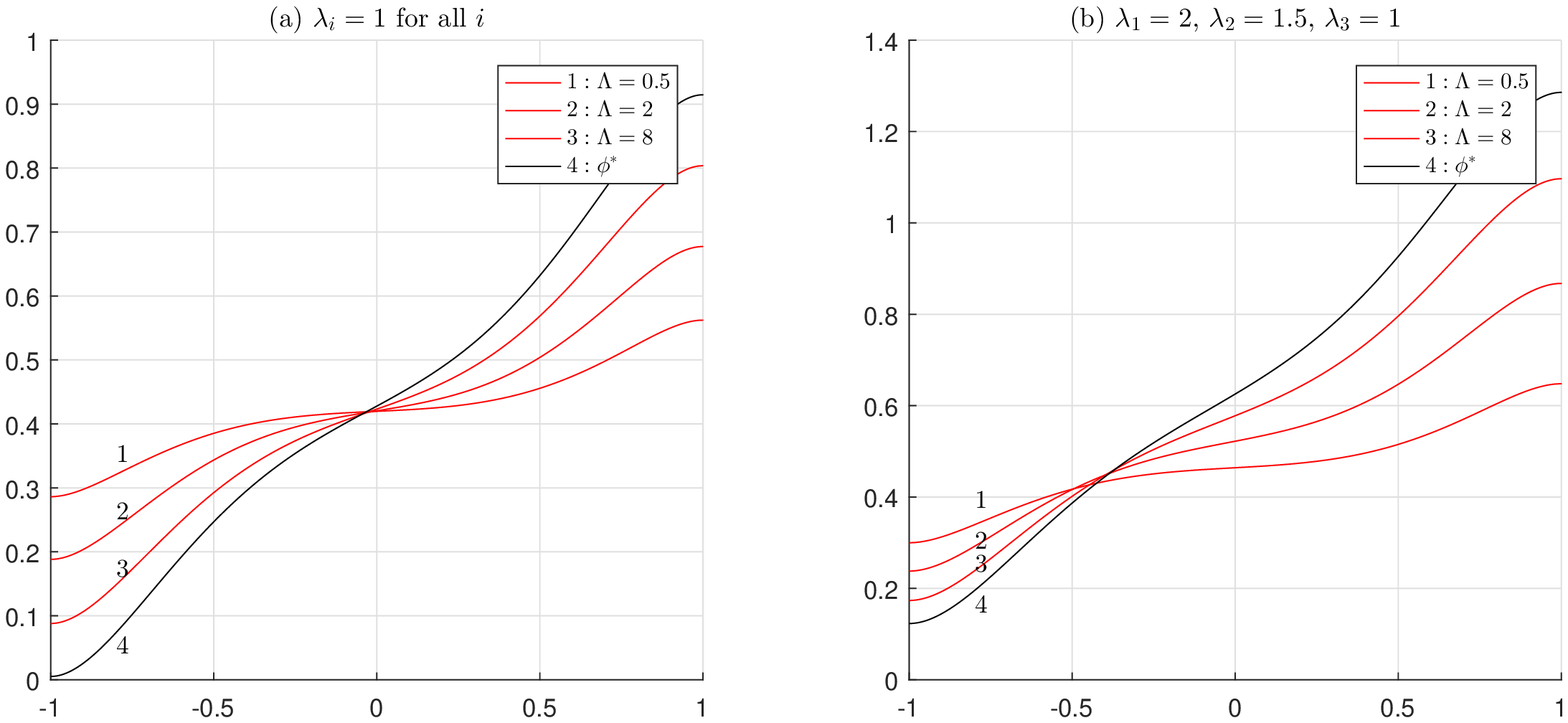}
\caption{The numerical profiles of $\phi_\Lambda$ and $\phi^*$ with the Neumann boundary condition, where curves 1--3 are the profiles of $\phi_\Lambda$ with $\Lambda=0.5,2,8$, and curve 4 is the profile of $\phi^*$.}
\label{fig4}
\end{figure}

\newpage

\noindent{\large\bf Final Remark. (Poisson--Boltzmann approach)}

Instead of steady state PNP-steric equations, we provide another approach to obtain equations \eqref{eq:1.06} and \eqref{eq:1.07} which can be derived by the following energy functional.
\[E[c_0,c_1,\dots,c_N,\phi]=E_{PB}[c_0,c_1,\dots,c_N,\phi]+E_{LJ}[c_0,c_1,\dots,c_N],\]
where
\[\begin{aligned}
&E_{PB}[c_0,c_1,\cdots,c_N,\phi]=\int_\Omega\left[-\frac12\epsilon\left|\nabla\phi\right|^2+\sum_{i=0}^Nc_i(\ln c_i-1)+\left(\rho_0+\sum_{i=1}^Nz_ic_i\right)\phi\right]\mathrm{d}x,\\
&E_{LJ}[c_0,c_1,\dots,c_N]=\frac12\sum_{i,j=0}^Ng_{ij}\int_\Omega c_i(x)c_j(x)\mathrm{d}x.\end{aligned}\]
Notice that $E_{PB}$ is the energy functional of conventional Poisson--Boltzmann equation with the form $-\nabla\cdot(\epsilon\nabla\phi)=\rho_0+\sum_{i=1}^Nz_ie^{\mu_i-z_i\phi}$ which can be obtained by $\delta E_{PB}/\delta\phi=0$ and $\delta E_{PB}/\delta c_i=\mu_i$ for $i=0,1,\dots,N$.
Besides, $E_{LJ}$ is the energy functional of the approximate Lennard--Jones potentials (cf. \cite{2014lin}), and one may derive the equations \eqref{eq:1.06} and \eqref{eq:1.07} by $\delta E/\delta\phi=0$ and $\delta E/\delta c_i=\mu_i$ for $i=0,1,\dots,N$.

~

\noindent{\large\bf Conclusion}

Modified Poisson-Boltzmann (mPB) equations play an important role to understand the steric effects of ion and solvent molecules.
To get such equations, we study steady state Poisson--Nernst--Planck equations with steric effects (PNP-steric equatons) and derive the Poisson--Boltzmann equation with steric effects (PB-steric equation) by the assumptions of steric effects.
Under the Robin (or Neumann) boundary condition, the PB-steric equation has a unique solution $\phi_\Lambda$, a parameter $\Lambda$, and positive constants $\lambda_i$'s which depend on the radii of ions and solvent moleclues.
As $\Lambda$ goes to infinity, the limiting equation of PB-steric equation becomes a mPB equation.
We firstly use the implicit function theorem on Banach spaces to prove the convergence of nonlinear term of PB-steric equation as $\Lambda$ tends to infinity.
Then we prove the uniform bound estimate of solution $\phi_\Lambda$ and obtain its convergence.
Numerical simulations are provided to support theoretical results. Further work is needed on the application of PB-steric equation.

~

\noindent{\bf Acknowledgement}

The research of T.C. Lin is partially supported by the National Center for Theoretical Sciences (NCTS) and MOST grant 109-2115-M-002 -003 of Taiwan. 

\section{Appendix I. Existence and uniqueness of \texorpdfstring{$\phi_\Lambda$}{}}

In this section, we fix $\Lambda>0$ and prove the existence and uniqueness of solution $\phi_\Lambda$ of equation \eqref{eq:1.12} with the Robin boundary condition and Neumann boundary condition \eqref{eq:1.08} and \eqref{eq:1.09}, respectively.
For the existence of solution $\phi_\Lambda$, we study the following energy minimization problems (N) and (R) for the Neumann and Robin boundary conditions, respectively.

\begin{itemize}
\item[(N)] Minimize $E_{eq}[\phi]$ subject to $\phi\in H^1(\Omega)$,
\item[(R)] Minimize $E[\phi]:=E_{eq}[\phi]+B_\eta[\phi]$ subject to $\phi\in\mathbb H_\eta(\Omega)$,
\end{itemize}
where the functionals
\[E_{eq}[\phi]=\frac12\int_\Omega\epsilon|\nabla\phi|^2\,\mathrm{d}x-\int_\Omega\rho_0\phi\,\mathrm{d}x-\int_\Omega F_\Lambda(\phi)\,\mathrm{d}x,\qquad B_\eta[\phi]=\begin{cases}
\frac1{2\eta}\int_{\partial\Omega}\epsilon(\phi-\phi_{bd})^2\,\mathrm{d}S_x&\text{if}~\eta>0,\\0&\text{if}~\eta=0,\end{cases}\]
and defined on the space
\[\mathbb{H}_\eta(\Omega):=\begin{cases}H^1(\Omega)&\text{if}~\eta>0,\\\{u\in H^1(\Omega):u-\phi_{bd}\in H_0^1(\Omega)\}&\text{if}~\eta=0.\end{cases}\]
Here the function $F_\Lambda(\phi)=\int_0^{\phi}f_\Lambda(s)\,\mathrm{d}s$, and the function $f_\Lambda=f_\Lambda(\phi)$ is defined in Proposition~\ref{prop:2.2} which gives $F_\Lambda''(\phi)=f_\Lambda'(\phi)<0$ for $\phi\in\mathbb R$.
Hence we may apply the Direct method (cf. \cite{2008struwe}) to solve problems (N) and (R).
The argument of problem (N) is similar to that of problem (R) so we omit here and only provide the argument of problem (R) in the rest of this section.

To apply the Direct method on problem (R), we need the following lemma.

\begin{lma}
\label{lma:6.1}
Functional $E$ is coercive on $H^1(\Omega)$ for $\eta>0$.
\end{lma}
\begin{proof}
By Young's inequality, we have
\begin{equation}
\label{eq:6.01}
B_{\eta}\left[\phi\right]\geq\frac1{4\eta}\int_{\partial\Omega}\epsilon\phi^2\,\mathrm{d}S_x-\frac1{2\eta}\int_{\partial\Omega}\epsilon\phi_{bd}^2\,\mathrm{d}S_x.
\end{equation}
By Propositions~\ref{prop:2.2} and \ref{prop:2.4}, function $F_\Lambda$ is strictly concave and has an absolute maximum denoted by $M_F$.
Hence by \eqref{eq:6.01}, we obtain
\begin{equation}
\label{eq:6.02}
E[\phi]\geq C_\eta\left(\int_{\Omega}\left|\nabla\phi\right|^2\mathrm{d}x+\int_{\partial\Omega}\phi^2\,\mathrm{d}S_x\right)-\int_\Omega|\rho_0\phi|\,\mathrm{d}x-\frac1{2\eta}\int_{\partial\Omega}\epsilon\phi_{bd}^2\,\mathrm{d}S_x-M_F\left|\Omega\right|,
\end{equation}
where $C_{\eta}=\frac1{4\eta}\min\left\{1,2\eta\right\}\min_{\overline{\Omega}}\epsilon>0$ and $\left|\Omega\right|$ denotes the Lebesuge measure of $\Omega$.
Moreover, for any $\phi\in H^1(\Omega)$, Friedrichs' inequality gives
\begin{equation}
\label{eq:6.03}
\int_{\Omega}\left|\nabla\phi\right|^2\mathrm{d}x+\int_{\partial\Omega}\phi^2\,\mathrm{d}S_x\geq C_1\int_{\Omega}\phi^2\mathrm{d}x,\end{equation}
where $C_1$ is a positive constant depending only on the dimension $d$ and the measures of $\Omega$ and $\partial\Omega$.
Besides, Cauchy--Schwarz inequality gives
\begin{equation}
\label{eq:6.04}
\int_\Omega\left|\rho_0\phi\right|\,\mathrm{d}x\leq\left(\int_\Omega\rho_0^2\,\mathrm{d}x\right)^{1/2}\left(\int_\Omega\phi^2\,\mathrm{d}x\right)^{1/2}
\end{equation}
and by \eqref{eq:6.02}--\eqref{eq:6.04}, Lemma~\ref{lma:6.1} follows.
\end{proof}

By Lemma~\ref{lma:6.1}, we may prove the existence of minimizer of problem (R) as follows.
\begin{prop}
\label{prop:6.2}
Functional $E$ has a minimizer $\phi\in\mathbb{H}_\eta$ for any $\eta\geq0$.
\end{prop}
\begin{proof}
Suppose $\eta>0$ and $\mathbb H_\eta(\Omega)=H^1(\Omega)$.
By \eqref{eq:6.02}--\eqref{eq:6.04}, $\displaystyle\inf_{\phi\in H^1(\Omega)}E[\phi]$ exists so there exists a minimizing sequence $\{\phi_n\}_{n=1}^{\infty}\subsetneq H^1(\Omega)$ such that
\[\lim_{n\to\infty}E[\phi_n]=\inf_{\phi\in H^1(\Omega)}E[\phi]:=m_E.\]
Due to the coerciveness of $E$ (see Lemma~\ref{lma:6.1}), we have $\displaystyle\sup_{n\in\mathbb{N}}\left\|\phi_n\right\|_{H^1(\Omega)}<\infty$.
Along with \eqref{eq:6.02}, we can get $\displaystyle\sup_{n\in\mathbb{N}}\left\|\phi_n\right\|_{L^2(\partial\Omega)}<\infty$.
Thus, there exists a subsequence $\{\phi_{n_k}\}$ of $\{\phi_n\}$ and $\phi\in H^1(\Omega)$ such that $\phi_{n_k}\rightharpoonup\phi$ weakly in $H^1(\Omega)$ and $\phi_{n_k}\rightharpoonup\Gamma\phi$ weakly in $L^2(\partial\Omega)$ as $k\to\infty$, where $\Gamma\phi$ is the trace of $\phi$ on $\partial\Omega$.
Since $\phi_{n_k}\rightharpoonup\phi$ weakly in $H^1(\Omega)$ implies $\nabla\phi_{n_k}\rightharpoonup\nabla\phi$ weakly in $L^2(\Omega)$, then we obtain
\[
\liminf_{k\to\infty}\int_{\Omega}\left|\nabla\phi_{n_k}\right|^2\mathrm{d}x\geq\int_{\Omega}\left|\nabla\phi\right|^2\mathrm{d}x,\quad\liminf_{k\to\infty}\int_{\partial\Omega}\left|\phi_{n_k}-\phi_{bd}\right|^2\mathrm{d}S_x\geq\int_{\partial\Omega}\left|\Gamma\phi-\phi_{bd}\right|^2\mathrm{d}S_x,
\]
and
\[
\lim_{k\to\infty}\phi_{n_k}=\phi~~\text{a.e. in}~\Omega.\]
Along with the Fatou's lemma, we have
\[
\liminf_{k\to\infty}\int_{\Omega}-F_\Lambda(\phi_{n_k})\,\mathrm{d}x\geq\int_{\Omega}-F_\Lambda(\phi)\,\mathrm{d}x,
\]
which gives
\[m_E=\lim_{k\to\infty}E[\phi_{n_k}]\geq E[\phi]\geq m_E,\]
and $E$ attains its minimum $m_E$ at $\phi\in H^1(\Omega)=\mathbb{H}_\eta(\Omega)$ for $\eta>0$.
As $\eta=0$ (i.e., the Dirichlet boundary condition), $B_\eta[\phi]=0$ so we may use the similar argument to prove the existence of minimizer and complete the proof.\end{proof}

The minimizer of Proposition~\ref{prop:6.2} satisfies equation \eqref{eq:1.12} with the Robin boundary condition \eqref{eq:1.08} in weak sense.
Hereafter, for simplicity, we denote the solution as $\phi$ instead of $\phi_\Lambda$ before.
Now we prove the regularity of solution $\phi$ as follows.
Because $\phi$ is a minimizer of the functional $E$ on $H^1(\Omega)$, $\phi$ satisfies
\begin{equation}
\label{eq:6.05}
\int_{\Omega}\left[\epsilon\nabla\phi\cdot\nabla v-v\rho_0-vf_\Lambda(\phi)\right]\mathrm{d}x+\frac1{\eta}\int_{\partial\Omega}\epsilon\left(\phi-\phi_{bd}\right)v\,\mathrm{d}S_x=0
\end{equation}
for any $v\in H^1(\Omega)$.
Let $\Phi$ be the solution of the auxiliary Poisson equation
\begin{equation}
\label{eq:6.06}
-\nabla\cdot\left(\epsilon\nabla\Phi\right)=\rho_0+f_\Lambda(\phi)\quad\text{in}~\Omega
\end{equation}
with the Robin boundary condition $\Phi+\eta\frac{\partial\Phi}{\partial\nu}=\phi_{bd}$ on $\partial\Omega$.
We multiply equation \eqref{eq:6.06} by $(\Phi-\phi)$ and integrate it over $\Omega$.
Then we may use integration by parts to obtain
\begin{equation}
\label{eq:6.07}
\int_{\Omega}\epsilon\nabla\Phi\cdot\nabla(\Phi-\phi)\,\mathrm{d}x-\int_{\Omega}(\Phi-\phi)\left[\rho_0+f_\Lambda(\phi)\right]\mathrm{d}x+\frac1{\eta}\int_{\partial\Omega}\epsilon(\Phi-\phi_{bd})(\Phi-\phi)\,\mathrm{d}S_x=0.
\end{equation}

Here we have used the fact $\frac{\partial\Phi}{\partial\nu}=\frac1\eta(\phi_{bd}-\Phi)$ on $\partial\Omega$. Besides, we subtract \eqref{eq:6.05} with $v=\Phi-\phi$ from \eqref{eq:6.07} and get
\[\int_{\Omega}\epsilon\left|\nabla\left(\Phi-\phi\right)\right|^2\mathrm{d}x+\frac1{\eta}\int_{\partial\Omega}\epsilon(\Phi-\phi)^2\mathrm{d}S_x=0,\]
which implies that $\Phi\equiv\phi$ a.e. in $\Omega$.
Thus, by a standard bootstrap argument on equation \eqref{eq:6.06}, solution $\phi$ becomes a classical solution.
On the other hand, as $\eta=0$, we can use a similar argument to prove the regularity of solution $\phi$.

Now we prove the uniqueness of equation \eqref{eq:1.12} with the Robin boundary condition \eqref{eq:1.08}.
Suppose that $\phi_1,\phi_2\in\mathcal{C}^{\infty}\left(\Omega\right)\cap\mathcal{C}^2\left(\overline{\Omega}\right)$ are solutions of equation \eqref{eq:1.12} with the Robin boundary condition \eqref{eq:1.08}.
We subtract equation \eqref{eq:1.12} with $\phi=\phi_2$ from that with $\phi=\phi_1$.
Then we have
\begin{equation}
\label{eq:6.08}
-\nabla\cdot(\epsilon\nabla(\phi_1-\phi_2))=f_\Lambda(\phi_1)-f_\Lambda(\phi_2)\quad\text{in}~\Omega.
\end{equation}
Multiplying the equation \eqref{eq:6.08} by $\phi_1-\phi_2$ and integrating it over $\Omega$, we get
\[-\int_{\partial\Omega}(\phi_1-\phi_2)\nabla\cdot(\epsilon\nabla(\phi_1-\phi_2))\mathrm{d}x=\int_{\Omega}\left(\phi_1-\phi_2\right)[f_\Lambda(\phi_1)-f_\Lambda(\phi_2)]\,\mathrm{d}x.
\]
Then using integration by parts, we obtain
\begin{equation}
\label{eq:6.09}
\int_\Omega\epsilon|\nabla(\phi_1-\phi_2)|^2\mathrm{d}x+\eta\in_{\partial\Omega}\epsilon\left|\frac{\partial(\phi_1-\phi_2)}{\partial\nu}\right|^2\mathrm{d}S_x=\int_{\Omega}(\phi_1-\phi_2)[f_\Lambda(\phi_1)-f_\Lambda(\phi_2)]\mathrm{d}x.
\end{equation}
Here we have used the fact $\phi_1-\phi_2=-\eta\frac{\partial(\phi_1-\phi_2)}{\partial\nu}$ on $\partial\Omega$ which comes from the Robin boundary condition \eqref{eq:1.08}.
By Proposition~\ref{prop:2.2},
\begin{equation}
\label{eq:6.10}
\left(\phi_1-\phi_2\right)[f_\Lambda(\phi_1)-f_\Lambda(\phi_2)]\leq0\quad\text{in}~\Omega.
\end{equation}
Thus, by \eqref{eq:6.09} and \eqref{eq:6.10}, we have
\[
\int_{\Omega}\epsilon\left|\nabla\left(\phi_1-\phi_2\right)\right|^2\mathrm{d}x+\eta\int_{\partial\Omega}\epsilon\left|\frac{\partial\left(\phi_1-\phi_2\right)}{\partial\nu}\right|^2\mathrm{d}S_x\leq0,
\]
which implies that $\phi_1=\phi_2$ in $\overline{\Omega}$ so we complete the proof of uniqueness and conclude the following theorem.
\begin{thm}
\label{thm:6.3}
Under the same hypotheses in Theorem~\ref{thm:1.1}, equation \eqref{eq:1.12} with the Robin boundary condition \eqref{eq:1.08} has a unique solution $\phi_{\Lambda}\in\mathcal{C}^{\infty}\left(\Omega\right)\cap\mathcal{C}^2\left(\overline{\Omega}\right)$ for $\Lambda>0$.
\end{thm}

\section{Appendix II.}

In this section, we prove that equations \eqref{eq:1.14}--\eqref{eq:1.17} and equations \eqref{eq:1.19}--\eqref{eq:1.21} have the same form (up to scalar multiples).
Firstly, we recall these equations as follows.
\begin{align}
\label{eq:7.01}&\sum_{i=0}^N\lambda_ic_i^*(\phi^*)=\tilde{\mu}_0,\\
\label{eq:7.02}&c_i^*(\phi^*)=(c_0^*(\phi^*))^{\lambda_i/\lambda_0}\exp(\bar{\mu}_i-z_i\phi^*)\quad\text{for}~i=1,\dots,N,\\
\label{eq:7.03}&-\nabla\cdot(\epsilon\nabla\phi^*)=\rho_0+\sum_{i=1}^Nz_ic_i^*(\phi^*)\quad\text{in}~\Omega,
\end{align}
and
\begin{align}
\label{eq:7.04}&\sum_{i=0}^Nv_ic_i(\psi)=1,\\
\label{eq:7.05}&\frac{v_i}{v_0}\ln(v_0c_0)-\ln(v_ic_i)=\beta q_i\psi-\beta\mu_i\quad\text{for}~i=1,\dots,N,\\
\label{eq:7.06}&-\nabla\cdot(\epsilon\nabla\psi)=\rho_0+\sum_{i=1}^N\beta q_ic_i(\psi)\quad\text{in}~\Omega.
\end{align}
From \eqref{eq:7.01} and \eqref{eq:7.04}, we may assume that $\tilde{\mu}_0=1$.
Then we have $\lambda_i=v_i$, $c_i^*=c_i$ and $\phi^*=\psi$ for $i=0,1,\dots,N$.
Moreover, by taking the natural logarithm on both sides of \eqref{eq:7.02}, we have
\begin{equation}\label{eq:7.07}
\frac{\lambda_i}{\lambda_0}\ln c_0^*(\phi^*)-\ln c_i^*(\phi^*)=z_i\phi^*-\bar{\mu}_i\quad\text{for}~i=1,\dots,N.
\end{equation}
Comparing \eqref{eq:7.05} and \eqref{eq:7.07} with $\phi^*=\psi^*$, we get
\[z_i=\beta q_i,\qquad \bar{\mu}_i=\beta\mu_i+\frac{v_i}{v_0}\ln v_0-\ln v_i\quad\text{for}~i=1,\dots,N.\]
Hence it is clear that \eqref{eq:7.03} and \eqref{eq:7.06} are identical.
Therefore, equations \eqref{eq:7.01}--\eqref{eq:7.03} and equations \eqref{eq:7.04}--\eqref{eq:7.06} have the same form (up to scalar multiples).

\end{document}